\documentclass{amsart}
\usepackage{amssymb,amstext,amsmath,amscd,amsthm,amsfonts,enumerate,graphicx,latexsym, comment}
\usepackage[usenames]{color}
\usepackage{tikz-cd}

\usepackage[all]{xy}

\usepackage{hyperref} 

\theoremstyle{plain}
\newtheorem{thm}{Theorem}[section]

\newtheorem*{thm*}{Theorem}
\newtheorem*{cor*}{Corollary}

\newtheorem{prop}[thm]{Proposition}

\newtheorem{lem}[thm]{Lemma}
\newtheorem{cor}[thm]{Corollary}

\newtheorem{claim}{Claim}
\newtheorem*{claim*}{Claim}

\theoremstyle{definition}
\newtheorem{defn}[thm]{Definition}

\newtheorem{ex}[thm]{Example}

\newtheorem{rem}[thm]{Remark}

\newtheorem{fact}[thm]{Fact}

\newtheorem{setup}[thm]{Setup}
\newtheorem*{acknowledgments}{Acknowledgments}

\theoremstyle{remark}

\newtheorem*{proof of claim}{{\sl Proof of Claim}}

\numberwithin{equation}{thm}

\def\Ext{\operatorname{Ext}}

\def\Im{\operatorname{Im}}

\def\Ker{\operatorname{Ker}}

\def\bbZ{\mathbb{Z}}

\def\Hom{\operatorname{Hom}}

\def\Max{\operatorname{Max}}

\def\Ker{\mathrm{Ker}}
\def\Im{\mathrm{Im}}

\def\tr{\mathrm{tr}}

\def\m{\mathfrak m}
\def\n{\mathfrak n}

\newcommand{\Ann}{\mathrm{Ann}}

\newcommand{\rme}{\mathrm{e}}

\newcommand{\rmr}{\mathrm{r}}

\newcommand{\rmK}{\mathrm{K}}

\newcommand{\rmM}{\mathrm{M}}

\newcommand{\rmQ}{\mathrm{Q}}

\newcommand{\fkb}{\mathfrak{b}}

\newcommand{\fkm}{\mathfrak{m}}
\newcommand{\fkn}{\mathfrak{n}}

\newcommand{\fkp}{\mathfrak{p}}
\newcommand{\fkq}{\mathfrak{q}}

\newcommand{\fkM}{\mathfrak{M}}

\newcommand{\mapright}[1]{%
\smash{\mathop{%
\hbox to 1cm{\rightarrowfill}}\limits^{#1}}}

\newcommand{\mapleft}[1]{%
\smash{\mathop{%
\hbox to 1cm{\leftarrowfill}}\limits_{#1}}}

\def\depth{\operatorname{depth}}

\def\Ass{\operatorname{Ass}}
\def\Spec{\operatorname{Spec}}

\def\ol{\overline}

\def\tr{\mathrm{tr}}

\title[]{Characterization of almost Gorenstein rings in terms of the trace ideal}

\author{Ryotaro Isobe}
\address{Department of Mathematics and Informatics, Graduate School of Science, Chiba University, Yayoi-cho 1-33, Inage-ku, Chiba, 263-8522, Japan}
\email{ryotaro.isobe@faculty.gs.chiba-u.jp}

\author{Shinya Kumashiro}
\address{Department of Mathematics, Osaka Institute of Technology, 5-16-1 Omiya, asahi-ku, Osaka, 535-8585, Japan}
\email{shinya.kumashiro@oit.ac.jp}

\thanks{R. Isobe was supported by JSPS KAKENHI Grant Number JP24K16910. S. Kumashiro was  supported by JSPS KAKENHI Grant Number JP24K16909 and by Grant for Basic Science Research Projects from the Sumitomo Foundation (Grant number 2200259).}

\begin{document}

\begin{abstract}
We provide a characterization of one-dimensional almost Gorenstein rings in terms of the trace ideal. As an application, we investigate the almost Gorenstein property of certain $\mathbb{Z}_2$-graded rings.

\end{abstract}

\maketitle

\section{Introduction}\label{section1}

The notion of almost Gorenstein rings was introduced by V. Barucci and R. Fr\"{o}berg \cite{BF} in Cohen-Macaulay analytically unramified local rings of dimension one. After their work, S. Goto, N. Matsuoka, and T. T. Phuong \cite{GMP} stretched the notion in arbitrary Cohen-Macaulay local rings of dimension one. 
Nowadays, the study of almost Gorenstein rings is explicated in arbitrary Cohen-Macaulay local rings (\cite{GTT}), and we can regard almost Gorenstein rings as generalized Gorenstein rings with respect to the maximal ideal in the sense of \cite{GK}.

In this article, we provide a new characterization of almost Gorenstein rings of dimension one in terms of the trace ideal, and apply it to analyze the almost Gorenstein property of a $\mathbb{Z}_2$-graded ring of the form $R\oplus \m$ with the grading $\deg R=0$ and $\deg \fkm =1$.
Let $(R, \m)$ be a Cohen-Macaulay local ring of dimension one having the canonical module $\rmK_R$. 
$R$ is called {\it an almost Gorenstein ring} if $R$ has an ideal $I$ such that $I\cong \rmK_R$ and $\rme_{I}^1(R)\le \rmr(R)$, where $\rme_{I}^1(R)$ denotes the first Hilbert coefficient of $I$. Recall that for an $R$-module $M$,
\[
\tr_R(M):= \sum_{f\in \Hom_R(M, R)} \Im f
\]
is called the {\it trace ideal} of $M$ (\cite{Lin}). 
The first main result of this article is as follows.

\begin{thm}  {\rm (Theorem \ref{p52} and Remark \ref{rem28})}\label{1.1}
Suppose that $(R, \m)$ is a Cohen-Macaulay local ring of dimension one having the canonical module $\rmK_R$. Then the following conditions are equivalent. 
\begin{enumerate}[$(1)$]
\item
$R$ is an almost Gorenstein ring.
\item
$\tr_R(\fkm \rmK_R)\supseteq \fkm$.
\end{enumerate}

\end{thm}

The other results concern  $\mathbb{Z}_2$-graded rings.
A ring $A$ is called a {\it $\mathbb{Z}_2$-graded ring} if $A$ has a decomposition $A=A_0 \oplus A_1$ as an additive group such that $A_{\overline{i}} A_{\overline{j}} \subseteq A_{\overline{i+j}}$ for all $\ol{i}, \ol{j}\in \mathbb{Z}_2=\{\ol{0}, \ol{1}\}$.
One of the simplest classes of commutative $\mathbb{Z}_2$-graded rings is the idealizations. For a commutative ring $R$ and an $R$-module $M$, the {\it idealization} $R\ltimes M$ of $M$ is a commutative ring defined by the additive group $A=R\oplus M$ with the product 
\[
(a,x)(b,y):=(ab, ay+bx) 
\]
for $a,b\in R$ and $x,y\in M$. The notion of the idealization has many applications and well considered. For instance, one can find a lot of papers citing \cite{AW}. We can regard the idealizations as $\mathbb{Z}_2$-graded rings by a natural grading with $\deg R=0$ and $\deg M=1$. The point is that all the products of elements of degree one are zero. As another example of $\mathbb{Z}_2$-graded rings, one can find finite extensions $R[X]/(X^2-a)=R\oplus RX$ of rings $R$, where $R[X]$ is the polynomial ring over $R$ and $a\in R$. 
Furthermore, any $\mathbb{Z}$-graded rings can be regarded as $\mathbb{Z}_2$-graded rings by reading the grading modulo $2$.

As stated in \cite{C}, the structure of $\mathbb{Z}_2$-graded rings $A=A_0\oplus A_1$ is given by a commutative ring $R=A_0$, an $R$-module $M=A_1$, and the product of elements of degree one 
\[
\varphi: M\times M \to R.
\]
Therefore, a $\mathbb{Z}_2$-graded ring $A=R\oplus M$ is explored via the triad $(R, M, \varphi)$. Since $A$ becomes the idealization if $\varphi=0$, in this article, we denote by $R\times_\varphi M$ the $\mathbb{Z}_2$-graded rings and call it the {\it idealization of $M$ with respect to $\varphi$} (Definition \ref{def1}). 

Let $(R, \m)$ be a Cohen-Macaulay local ring of dimension one. 
The second main result of this article is the following theorem, which characterizes the almost Gorenstein property of $A:=R\times_{\varphi} \m$. 
In this case, $\varphi$ can be expressed by $\alpha \in \fkm:\fkm^2 (\subseteq \rmQ(R))$ as $\varphi : \m\times \m \to R; (x, y)\mapsto \alpha xy$. 
Let $B=\m:\m=\{ \beta\in \rmQ(R) \mid \beta \m \subseteq \m \}$ be a module-finite birational extension of $R$. 

\begin{thm} {\rm (Theorem \ref{t55})}\label{1.2}
Suppose that $(R, \m)$ is a Cohen-Macaulay local ring of dimension one having an ideal $I$ such that $I\cong \rmK_R$ and $R$ is not a discrete valuation ring. Consider the following conditions.
\begin{enumerate}[\rm(1)] 
\item $A$ is an almost Gorenstein ring.
\item $R$ is an almost Gorenstein ring and $\tr_B(\langle 1, \alpha \rangle_B)=B$, where $\langle 1, \alpha \rangle_B$ denotes the $B$-module generated by $1$ and $\alpha$.
\item $R$ is an almost Gorenstein ring and either $\alpha\in B$ or $\alpha^{-1}\in B$.
\end{enumerate}
Then, {\rm (1)$\Leftrightarrow$(2)$\Leftarrow$(3)} hold. {\rm (2)$\Rightarrow$(3)} also holds if $B$ is a local ring.
\end{thm}

In the case where $\alpha=0$, that is, $A$ is the idealization in the sense of \cite{AW}, it is known that $A$ is almost Gorenstein if and only if $R$ is almost Gorenstein (\cite[Theorem 6.5]{GMP}). Our result generalizes this to arbitrary $\alpha\in \rmQ(R)$.

In what follows, we explain how this article is organized. 
Let $R$ be a commutative ring and $M$ an $R$-module. Let $A=R\times_{\varphi} M$, where $\varphi$ is a defining map of $A$.

In Section \ref{s1}, we give the proof of Theorem \ref{1.1}.
In Section \ref{s2}, we explore the structure of $\mathbb{Z}_2$-graded rings and prepare several propositions that we need later.
In Section \ref{s3}, we characterize local, Noetherian, Artinian, and Cohen-Macaulay properties of $R\times_\varphi M$ in terms of structures of $R$, $M$, and $\varphi$. 
Although some of the results in Sections \ref{s2} and \ref{s3} are essentially known (cf. \cite{C}), we reestablish them here using our own method for the sake of completeness.
In Section \ref{s5}, we give the proof of Theorem \ref{1.2}.

Sections \ref{s4} and \ref{s7} are presented as an appendix.
In Section \ref{s4}, we characterize the Gorenstein property of $A=R\times_{\varphi} M$. Although this has already been well established (see, for example, \cite{C, F, R}), we include an alternative and insightful proof for the sake of completeness.
In Section \ref{s7}, we explore the regularity of $A=R\times_{\varphi} M$. When $A$ is the idealization, $A$ is never regular. In this article, we show that there are many non-trivial examples such that $R\times_{\varphi} M$ is regular (Examples \ref{exr1}, \ref{exsec5}). In addition, we characterize the regularity of $A$ when $\dim A\le 2$ (Corollary \ref{p59}).

\begin{setup}
In this article, all rings are commutative. For a ring $R$, $\rmQ(R)$ (resp. $\ol{R}$ and $R^{\times}$) denotes the total ring of fractions of $R$ (resp. the integral closure of $R$ and the set of unit elements of $R$). For an $R$-module $M$, $\ell_R(M)$ denotes the length of $M$.

We denote by $(R,\fkm)$ a local ring  $R$ with the unique maximal ideal $\fkm$. Let $(R,\fkm)$ be a Noetherian local ring. For a finitely generated $R$-module $M$, $\mu_R(M)$ denotes the {\it number of minimal generators} of $M$. 
If $M$ is a Cohen-Macaulay $R$-module, $\rmr_R(M)$ denotes the {\it Cohen-Macaulay type} $\ell_R(\Ext_R^{t}(R/\fkm, M))$ of $M$, where $t=\dim M$. $v(R)$ denotes the {\it embedding dimension} $\mu_R(\fkm)$ of $R$. 

A finitely generated $R$-submodule of $\rmQ(R)$ containing a non-zerodivisor of $R$ is called a {\it fractional ideal} of $R$. 
For fractional ideals $X, Y$, let $X:Y=\{ \alpha\in \rmQ(R) \mid \alpha Y\subseteq X \}$ denote the colon fractional ideal. It is well-known that $X: Y\cong \Hom_R(Y, X)$ with the correspondence $\alpha\mapsto ({\cdot}\alpha: Y \to X; y\mapsto \alpha y)$ (\cite{HK}). 

For ideals $I$ and $J$, $I:_R J=\{ a\in R \mid a J\subseteq I \}$ denotes the colon ideal of $R$.

\end{setup}

\section{Characterization of almost Gorenstein rings}\label{s1}

In this section, we characterize one-dimensional almost Gorenstein rings in terms of the trace ideal $\tr_R(\fkm \rmK_R)$. 
Here, we focus on the case of dimension one, thus let us recall the definition of almost Gorenstein rings in dimension one. 

\begin{defn} (\cite[Definition 3.1]{GMP})\label{def5.1}
Let $(R, \m)$ be a Cohen-Macaulay local ring of dimension one having the canonical module $\rmK_R$. 
$R$ is called {\it an almost Gorenstein ring} if $R$ has an ideal $I$ such that $I\cong \rmK_R$ and $\rme_{I}^1(R)\le \rmr(R)$, where $\rme_{I}^1(R)$ denotes the first Hilbert coefficient of $I$.
\end{defn}

We note that $\rme_{I}^1(R)$ is independent of the choice of canonical ideals $I\subsetneq R$  (\cite[Corollary 2.13]{GMP}). 
Throughout this section, let $(R, \m)$ be a Cohen-Macaulay local ring of dimension one having the canonical module $\rmK_R$. Suppose that there exists a canonical ideal, that is, an ideal which is isomorphic to $\rmK_R$. We then refer several results of \cite{Kum} to avoid assuming that $R/\fkm$ is infinite. By  \cite[Corollary 3.5]{Kum}, we can choose a canonical ideal $\omega$ and $a\in R$ such that $\omega^{\ell+1}=a\omega^\ell$ and $(a)^{\ell}\subseteq \omega^\ell$ for some $\ell\gg 0$ without assuming that $R/\fkm$ is infinite (\cite[Corollary 3.5]{Kum}). Set fractional ideals 
\begin{center}
$B:=\fkm:\fkm$, \quad $K:=\frac{\omega}{a}$, \quad and $S:=R[K]=K^\ell$ 
\end{center}
for $\ell\gg 0$  (see \cite[Proposition 2.4]{Kum}). With this notations, the following equivalent conditions of the almost Gorenstein property are known.

\begin{fact}{\rm (\cite[Proposition 3.10]{Kum})}\label{f52}
The following conditions are equivalent.
\begin{enumerate}[\rm(1)] 
\item $R$ is an almost Gorenstein ring. 
\item $\fkm K\subseteq R$.
\item $\fkm S=\fkm$.
\end{enumerate}
\end{fact}

\begin{fact}{\rm (\cite[Theorem 3.16]{GMP})}\label{ff53}
The following conditions are equivalent.
\begin{enumerate}[\rm(1)] 
\item $R$ is an almost Gorenstein ring but not a Gorenstein ring. 
\item $R$ has an ideal $I$ such that $I\cong \rmK_R$ and $\rme_{I}^1(R)= \rmr(R)$.
\end{enumerate}
\end{fact}

\begin{rem}\label{rr53}
Fact \ref{f52} is well-known if $\omega$ has a reduction $(a)$ (\cite[Theorem 3.11]{GMP}). 
On the other hand, there exists an example of an almost Gorenstein ring with a canonical ideal $\omega$ such that $\omega$ has no reduction (\cite[Remark 2.10]{GMP}). 
\end{rem}

To give more equivalent conditions of the almost Gorenstein property, we prepare the following. 

\begin{lem}\label{ll55}
Let $X$ be a fractional ideal of $R$. Then, $X\fkm =\fkm$ if and only if $XB=B$.
\end{lem}

\begin{proof}
(``if'' part): We note that $\fkm B=\fkm$. Hence, the assertion follows from the equations $\fkm X=\fkm BX=\fkm B=\fkm$.

(``only if '' part): Since $X\subseteq \fkm:\fkm = B$, we have $XB\subseteq B$. Assume that $XB\subsetneq B$. There exists a maximal ideal $\fkM$ of $B$ such that $XB\subseteq \fkM$. 
By localizing $(XB) \fkm=X(\fkm B)=X\fkm = \fkm$ at $\fkM$, we have $(XB)_\fkM{\cdot}\fkm B_\fkM =\fkm B_\fkM$. It follows that $\fkm B_\fkM=0$ by Nakayama's lemma. This is a contradiction since $\fkm=\fkM\cap R$.
\end{proof}

We note a fact on trace ideals which we use in this section. 

\begin{fact}\label{lll55}
Let $(R, \fkm)$ be a Noetherian local ring. Let $I, J$ be fractional ideals of $R$. The following hold true.
\begin{enumerate}[\rm(1)] 
\item (\cite[Corollary 2.2]{GIK2}): $\tr_R(I)=(R:I)I$.
\item (\cite[Proposition 1.4]{HHS}): $\tr_R(IJ)\subseteq \tr_R(I) \cap \tr_R(J)$.
\item (\cite[Proposition 2.2]{Kum2}): $\tr_R(\fkm)=\fkm$ if and only if $R$ is not a discrete valuation ring.
\end{enumerate}
\end{fact}

%
%

We give several equivalent conditions of the almost Gorenstein property other than Fact \ref{f52} as follows. 
We should note that the equivalence of the conditions (1)-(3) of Theorem \ref{p52} is known when $(a)$ is a reduction of $\omega$ (\cite[Theorem 3.11]{GMP}, \cite{Kob}).
We note that if $(a)$ is not a reduction of $\omega$, we cannot say that $R\subseteq K$. This makes the proof difficult.

\begin{thm}\label{p52}
The following conditions are equivalent.
\begin{enumerate}[\rm(1)] 
\item $R$ is an almost Gorenstein ring. 
\item $\fkm K =\fkm$.
\item $\fkm \rmK_R \cong \fkm$.
\item $\tr_R(\fkm \rmK_R)\supseteq \fkm$.
\item There exists a fractional ideal $X$ of $R$ such that $X\fkm K=\fkm$.
\end{enumerate}
\end{thm}

\begin{proof}
(2)$\Rightarrow$(3) and (3)$\Rightarrow$(4) are clear. 

(4)$\Rightarrow$(5): If $R$ is a discrete valuation ring, then $K=R$. Hence, we can choose $R$ as $X$. Suppose that $R$ is not a discrete valuation ring. Then, $\fkm\subseteq \tr_R(\fkm \rmK_R)\subseteq \tr_R(\fkm)=\fkm$ by Fact \ref{lll55}(2), (3). Thus, $\tr_R(\fkm \rmK_R)=\fkm$. 
On the other hand, we have $\tr_R(\fkm \rmK_R)= (R:\fkm K)\fkm K$ by Fact \ref{lll55}(1). Thus, we can choose $R:\fkm K$ as $X$.

(5)$\Rightarrow$(1): We may assume that $R$ is not a discrete valuation ring. Since $X\fkm K= \fkm$, we obtain that
\begin{align}\label{eq1111}
\fkm = X \fkm K = X^2 \fkm K^2 = \cdots = X^\ell \fkm K^\ell = X^\ell \fkm S
\end{align}
for $\ell \gg 0$. Thus, $X^\ell \subseteq \fkm :\fkm S$. By \eqref{eq1111}, we obtain that 
\[
\fkm = X^\ell \fkm S \subseteq (\fkm :\fkm S) \fkm S \subseteq \fkm.
\]
Hence, $\fkm =\fkm (\fkm :\fkm S)S=\fkm (B:S)S=\fkm (B:S)$ since $\fkm :\fkm S=B:S$ and $B:S$ is an ideal of $S$ and $B$ (see, for example, \cite[Excercises 2.11]{HS}). 
By Lemma \ref{ll55}, we have $B=(B:S)B=B:S$. Therefore, $S\subseteq B:B=B$, that is, $\fkm K\subseteq \fkm S\subseteq \fkm \subseteq R$. It follows that $R$ is almost Gorenstein by Fact \ref{f52}.

(1)$\Rightarrow$(2): Suppose that $R$ is an almost Gorenstein ring. We may assume that $R$ is not Gorenstein. By Fact \ref{f52}, we have $\fkm K\subseteq R$, and thus we have the following diagram of inclusions:
\[
\xymatrix{
& S   &  \\
R \ar@{-}[ur] &  & K \ar@{-}[ul]\\
& \fkm K \ar@{-}[ur] \ar@{-}[ul] 
}
\]
We note that $\ell_R(S/R)=\rme_{\omega}^1(R)$ by \cite[proof of Proposition 3.7(a)]{Kum} and $\ell_R(K/\fkm K)=\rmr(R)$. By considering the $K$-dual $K:-=\Hom_R(-, K)$ of the exact sequence $0 \to K \to S \to S/K\to 0$, we have $\Ext_R^1(S/K, K) \cong (K:K)/(K:S) = R/(K:S)$. It follows that $\ell_R(S/K)=\ell_R((K:K)/(K:S))=\ell_R(R/(K:S))$. On the other hand, since $\fkm S=\fkm$ by Fact \ref{f52}, we have $\fkm S = \fkm K^{\ell+1} = \fkm SK= \fkm K\subseteq K$, that is, $\fkm \subseteq K:S$. It follows that $\ell_R(S/K)=\ell_R(R/(K:S))\le 1$. Therefore, since $R$ is non-Gorenstein almost Gorenstein, by Fact \ref{ff53} we have 
\begin{align*}
\ell_R(R/\fkm K) =&\ell_R(S/K) + \ell_R(K/\fkm K) - \ell_R(S/R) \\
=&\ell_R(S/K) + \rmr(R) - \rme_{\omega}^1(R) \\
=&\ell_R(S/K) \le 1.
\end{align*}
Since $\fkm K=R$ is impossible by \cite[Lemma 3.9]{Kum}, we obtain that $\fkm K=\fkm$.
\end{proof}

\begin{rem}\label{rem28}
The equivalence (1)$\Leftrightarrow$(4) of Theorem \ref{p52} holds under the assumption that $R$ is a Cohen-Macaulay local ring of dimension one having the canonical module $\rmK_R$, that is, we need not assume the existence of an ideal $I$ such that $I\cong \rmK_R$.  Indeed, this condition is automatically satisfied under each assumption (1) and (4). 
If we assume (1), then such an ideal $I$ exists by the definition of almost Gorenstein rings (Definition \ref{def5.1}). Suppose that (4). Then, by \cite[Lemma 2.1]{HHS}, $R_\fkp$ is Gorenstein for all $\fkp\in \Ass (R)$ since $R_\fkp = (\tr_R(\fkm \rmK_R))_\fkp=\tr_{R_\fkp}((\fkm \rmK_R)_\fkp) = \tr_{R_\fkp}(\fkm{R_\fkp} \cdot \rmK_{R_\fkp})= \tr_{R_\fkp}(\rmK_{R_\fkp})$ for all $\fkp\in \Ass (R)$. 
By \cite[Proposition 3.3.18]{BH}, it follows that there exists an ideal $I$ such that $I\cong \rmK_R$. 
\end{rem}

\begin{rem}
The equivalence (1)$\Leftrightarrow$(4) of Theorem \ref{p52} gives a relation between the almost Gorenstein property and the nearly Gorenstein property. Recall that for arbitrary Cohen-Macaulay local ring $(R, \fkm)$ possessing the canonical module $\rmK_R$, $R$ is called {\it nearly Gorenstein} if $\tr_R(\rmK_R)\supseteq \fkm$ (\cite[Definition 2.2]{HHS}). Since $\tr_R(\fkm \rmK_R)\subseteq \tr_R(\rmK_R)$  by Fact \ref{lll55}(2), we recover a result of Herzog-Hibi-Stamate (\cite[Proposition 6.1]{HHS}) saying that one-dimensional almost Gorenstein rings are nearly Gorenstein.
\end{rem}

\section{The structure of $\mathbb{Z}_2$-graded rings}\label{s2}

In this section, we summarize the structure and some properties of  $\mathbb{Z}_2$-graded rings that we need later. Although some of the results in this section are essentially known (cf. \cite{C}), we reestablish them using our own method for the sake of completeness.
Let us begin with the definition of $\mathbb{Z}_2$-graded rings. 

\begin{defn}
We say that a ring $A$ is a {\it $\mathbb{Z}_2$-graded ring} if $A$ has a decomposition $A=A_0 \oplus A_1$ as an additive group such that $A_{\overline{i}} A_{\overline{j}} \subseteq A_{\overline{i+j}}$ for all $\ol{i}, \ol{j}\in \mathbb{Z}_2=\{\ol{0}, \ol{1}\}$.
\end{defn}

Here are some quick examples of $\mathbb{Z}_2$-graded rings. Many more examples follow from our results later.

\begin{ex}
Let $R$ be a ring, $M$ be an $R$-module, and $K$ be a field. $R[X, Y, Z]$ denote the polynomial ring over $R$.
The following rings $A$ have the structure of $\mathbb{Z}_2$-graded rings. 
\begin{enumerate}[\rm(1)] 
\item $A=R[X]/(X^2-a)$, where $a\in R$. In particular, $\mathbb{C}=\mathbb{R}\oplus \mathbb{R}i\cong \mathbb{R}[X]/(X^2+1)$ is a $\mathbb{Z}_2$-graded ring.
\item (Idealization): $A=R\oplus M$ with the product $(a, x)(b,y)=(ab, ay+bx)$ for $a,b\in R$ and $x,y\in M$. 
\item (Example \ref{aaa48}): Let $a,b,c\in K$. Set
\[
A = K[X, Y, Z]/(X^2, XY, XZ, Y^2 - aX, YZ - bX, Z^2 - cX).
\] 
Then, $A$ is an Artinian $\mathbb{Z}_2$-graded local ring with the grading $\deg X=0$ and $\deg Y=\deg Z =1$. Moreover, $A$ is Gorenstein if and only if $ac\ne b^2$. 
\end{enumerate}
\end{ex}


The following is a construction of $\bbZ_2$-graded rings. We also see that all $\bbZ_2$-graded rings can be obtained by this construction (Theorem \ref{structure}).

\begin{defn}\label{def1}
Let $R$ be a ring and $M$ be an $R$-module. Let $\varphi: M\times M \to R$ be an $R$-bilinear homomorphism satisfying the following conditions.
\begin{enumerate}[\rm(1)] 
\item $\varphi(x, y)=\varphi(y,x)$ for all $x,y\in M$.
\item $\varphi(x,y)z=\varphi(y,z)x$ for all $x,y,z\in M$.
\end{enumerate}
Then an additive group $A=R\oplus M$ can be regarded as a $\bbZ_2$-graded ring by the  multiplication
\[
(a,x){\cdot}(b,y):=(ab+\varphi(x,y), ay+bx),
\]
where $a,b\in R$ and $x,y\in M$. We denote the above $\bbZ_2$-graded ring by $R\times_\varphi M$ and call it {\it the idealization of $M$ with respect to $\varphi$}.
\end{defn}

\begin{rem}\label{r25}
If $\varphi=0$, then $R\times_\varphi M$ is exactly the same as the idealization of $M$ in the sense of \cite{AW}. 
\end{rem}

\begin{thm}\label{structure} 
Let $A=A_0\oplus A_1$ be a $\bbZ_2$-graded ring. Set 
\[
\varphi: A_1\times A_1 \to A_0; \ \ (x,y) \mapsto x{\cdot}y
\]
for $x,y\in M$. Then $A=A_0 \times_\varphi A_1$.
\end{thm}

\begin{proof}
Let $\alpha, \beta\in A$. We write $\alpha = a+x$ and $\beta=b+y$ where  $a, b\in A_0$ and $x, y\in A_1$.
Then, 
$$\alpha{\cdot}\beta = (a+x){\cdot}(b+y)=(ab+\varphi(x, y))+(ay+bx).$$
Thus, it is enough to show that $\varphi$ satisfies conditions $(1)$ and $(2)$ of Definition \ref{def1}. This is clear because $\varphi(x, y)=x{\cdot}y$ and $A$ is a commutative ring. 
\end{proof}


By virtue of Theorem \ref{structure}, we explore $\mathbb{Z}_2$-graded rings $A=R\oplus M$ via the triad $(R, M, \varphi)$. 
\begin{setup}\label{setup2.5}
Let $R$ be a ring, $M$ be an $R$-module, and $\varphi: M\times M \to R$ be an $R$-bilinear  homomorphism satisfying the following conditions.
\begin{enumerate}[\rm(1)] 
\item $\varphi(x, y)=\varphi(y,x)$ for all $x,y\in M$.
\item $\varphi(x,y)z=\varphi(y,z)x$ for all $x,y,z\in M$.
\end{enumerate} 

Set $A=R\times_\varphi M$.
\end{setup}

We summarize the fundamental properties of $\mathbb{Z}_2$-graded ideals of $A$. We note that the case of idealizations are known in \cite[Theorem 3.3]{AW}. 

\begin{prop}\label{graded}
\begin{enumerate}[\rm(1)] 
\item The $\mathbb{Z}_2$-graded ideals of $A$ have the form $I\times N$, where $I$ is an ideal of $R$ and $N$ is an $R$-submodule $N$ of $M$ such that $\varphi(M, N)\subseteq I$ and $IM\subseteq N$.
\item 
Let $I$ be an ideal of $R$ and $N$ an $R$-submodule of $M$. Then the graded ideal of $A$ generated by $I\times 0$ and $0\times N$ are $[I + \varphi(M,N)]\times [IM + N]$.
\item 
Suppose that $J_1=I_1\times N_1$ and $J_2=I_2\times N_2$ are ideals of $A$. Then 
\begin{align*}
J_1+J_2&=[I_1+I_2]\times [N_1+ N_2]\\ 
J_1\cap J_2&=[I_1\cap I_2]\times [N_1\cap N_2]\\
J_1{\cdot}J_2&=[I_1I_2 + \varphi(N_1, N_2)] \times [I_2 N_1 + I_1N_2].
\end{align*}
\item Let  $I$ be an ideal of $R$, and let $N$ be an $R$-submodule $N$ of $M$ such that $\varphi(M, N)\subseteq I$ and $IM\subseteq N$. Then $A/(I\times N) \cong (R/I)\times_{\ol{\varphi}} (M/N)$, where $\ol{\varphi}$ denotes the canonical map $\ol{\varphi}: M/N \times M/N \to R/I$ induced from $\varphi$.
\end{enumerate}
\end{prop}

\begin{proof}
This is straightforward to check.
\end{proof}

\begin{rem}
\begin{enumerate}[\rm(1)] 
\item Let $\varphi=0$, that is, $A=R\times_0 M$ is the idealization of $M$. Then, the ideal $0\times M$ of $A$ satisfies $(0\times M)^2=0$. It follows that $\Spec A = \{ \fkp\times M \mid \fkp \in  \Spec R  \}$ (\cite[Theorem 3.2]{AW}). This presentation  does not hold in general for $\mathbb{Z}_2$-graded rings (see, for example, Example \ref{ex48!}(1)).
\item In general, $0\times M$ is not an ideal of $A=R\times_\varphi M$ (see Proposition \ref{graded}(2)). In particular, the map $A\to R;(a,x)\mapsto a$ is not a  homomorphism of rings.
\end{enumerate}
\end{rem}



Let $K$ be an $R$-module. We next explore the $A$-action of an $R$-module 
\[
\Hom_R(A, K)\cong \Hom_R(M, K) \oplus K. 
\]
In Section \ref{s5} and Subsection \ref{subsection42}, we apply these results to construct the canonical module of $A$ from that of $R$ (Proposition \ref{p48}). 
Set $L:=\Hom_R(M, K) \oplus K$.
For $x\in M$ and $k\in K$, we set an $R$-linear homomorphism 
\[
\psi_{x, k}: M \to K;\  y\mapsto \varphi(x, y)k.
\]

\begin{prop}\label{dualaction}
Define the $A$-action of $L:=\Hom_R(M, K) \oplus K$ as
\[
(a, x){\cdot}(f, k):=(af + \psi_{x, k}, f(x) + ak)
\]
for $(a, x)\in A$ and $(f, k)\in L$. Then, $L$ is an $A$-module with the above action.
\end{prop}

\begin{proof}
Let $(a, x), (a_1, x_1), (a_2, x_2)\in A$ and $(f, k), (f_1, k_1), (f_2, k_2)\in L$. It is routine to check the assertions that 
\begin{align*}
[(a_1, x_1) + (a_2, x_2)]{\cdot}(f, k) &= (a_1, x_1){\cdot}(f, k) +  (a_2, x_2){\cdot}(f, k),\\
(a, x){\cdot}[(f_1, k_1) + (f_2, k_2)] &= (a, x){\cdot}(f_1, k_1)  + (a, x){\cdot}(f_2, k_2),\\
(1, 0){\cdot}(f, k)&=(f, k).
 \end{align*}
The rest is to prove that $(a_1, x_1)[(a_2, x_2){\cdot}(f,k)] = [(a_1, x_1)(a_2, x_2)]{\cdot}(f,k)$. The left hand side of the equation is 
\begin{align*}
(a_1(a_2 f + \psi_{x_2, k}) + \psi_{x_1, f(x_2) + a_2 k},\ (a_2f + \psi_{x_2, k})(x_1) + a_1(f(x_2) + a_2k)).
\end{align*}
On the other hand, the right hand side of the equation is 
\begin{align*}
((a_1a_2 + \varphi(x_1, x_2))f + \psi_{a_1x_2 + a_2 x_1, k},\ f(a_1x_2 + a_2x_1) + (a_1a_2 + \varphi(x_1, x_2))k).
\end{align*}
Hence, by noting that $\psi_{x_2, k}(x_1) = \varphi(x_2, x_1)k = \varphi(x_1, x_2)k$, we only need to check that 
\[
a_1 \psi_{x_2, k} + \psi_{x_1, f(x_2) + a_2 k} = \varphi(x_1, x_2)f + \psi_{a_1x_2 + a_2 x_1, k}.
\]
Let $y\in M$. Then, 
\begin{align*}
(a_1 \psi_{x_2, k} + \psi_{x_1, f(x_2) + a_2 k})(y)=& a_1\varphi(x_2, y)k + \varphi(x_1, y)(f(x_2)+a_2 k)\\
=&  \varphi(x_1, y)f(x_2) + (a_1\varphi(x_2, y) +  a_2\varphi(x_1, y))k\\
=&  f(\varphi(x_1, y)x_2) + \varphi(a_1x_2 + a_2 x_1, y)k\\
=&  f(\varphi(x_1, x_2)y) + \varphi(a_1x_2 + a_2 x_1, y)k\\
=&\varphi(x_1, x_2)f(y) + \varphi(a_1x_2 + a_2 x_1, y)k\\
=& (\varphi(x_1, x_2)f + \psi_{a_1x_2 + a_2 x_1, k})(y),
\end{align*}
where the fourth equality follows from Definition \ref{def1}(2). Therefore, the assertion holds.
\end{proof}

\begin{prop}\label{p33}
Let $L$ be an $A$-module with the action in Proposition \ref{dualaction}. Let $\iota: M \to A;\ x\mapsto (0, x)$ be an embedding. 
Then 
\[
\zeta: \Hom_R(A, K) \to L;\ \alpha \mapsto (\alpha\circ \iota, \alpha(1,0))
\]
is an isomorphism as $A$-modules.
\end{prop}

\begin{proof}
Since $\zeta$ is an isomorphism as $R$-modules, it is enough to show that 
\begin{align}\label{eq291}
(a, x)\cdot \zeta (\alpha)=\zeta ((a, x) \alpha)
\end{align}
for all $(a, x)\in A$ and $\alpha\in \Hom_R(A, K)$. Indeed, the left hand side of \eqref{eq291} is 
$$(a(\alpha \circ \iota)+\psi_{x, \alpha(1, 0)}, \alpha(a, x) ),$$
and the right hand side of \eqref{eq291} is 
$$(((a, x)\alpha)\circ \iota, \alpha (a, x)).$$
It is routine to check the assertion 
$$(a(\alpha \circ \iota)+\psi_{x, \alpha(1, 0)})(y)=(((a, x)\alpha)\circ \iota)(y)$$
for all $y\in M$.
\end{proof}

\section{The Noetherian, Artinian, and Cohen-Macaulay properties of $\mathbb{Z}_2$-graded local rings}\label{s3}
In this section we explore the Noetherian, Artinian, and Cohen-Macaulay properties of $\mathbb{Z}_2$-graded rings. 
As in Section \ref{s2}, the results presented in this section are essentially known (cf. \cite{AW, C}), but we reestablish them here using our own method.



\begin{rem}\label{integral}
$A$ is integral over $R$ because $(0, x)^2=(\varphi(x, x), 0) \in R\times 0$ for all $x\in M$. In particular, $\dim A=\dim R$ (see \cite[Exercise 9.2]{Mat}). 
\end{rem}

\begin{prop}\label{noether} 
The following are equivalent.
\begin{enumerate}[\rm(1)] 
\item $A$ is a Noetherian ring $($resp. an Artinian ring$)$.
\item $R$ is a Noetherian ring $($resp. an Artinian ring$)$ and $M$ is a finitely generated $R$-module. 
\end{enumerate}
\end{prop}

\begin{proof}
$(2) \Rightarrow (1)$: Suppose that $R$ is a Noetherian ring and $M$ is a finitely generated $R$-module. Then, $A$ is a module finite extension of $R$, so that $A$ is also a Noetherian ring. 
Suppose that $R$ is an Artinian ring and $M$ is a finitely generated $R$-module. Then $A$ is Noetherian and  $\dim A =\dim R=0$ by Remark \ref{integral}. Hence, $A$ is Artinian.

$(1) \Rightarrow (2)$: Suppose that $A$ is a Noetherian ring. Take a chain 
$$I_1\subseteq I_2 \subseteq \cdots \subseteq I_i \subseteq \cdots \subseteq R$$
of ideals of $R$. Note that $I_i A\cap R=I_i$ for each $i>0$ (see Proposition \ref{graded}(2)). Since $I_iA$ is an ideal of $A$, $I_i A=I_{i+1} A$ for all $i\gg 0$. It follows that $R$ is a Noetherian ring. 
Similarly, take a chain
$$M_1\subseteq M_2 \subseteq \cdots \subseteq M_i \subseteq \cdots \subseteq M$$
of $R$-submodules of $M$. Since $(0\times M_i) A=\varphi(M_i\times M)\times M_i$ for each $i>0$, we obtain the chain 
$$\varphi(M_1\times M)\times M_1\subseteq \varphi(M_2\times M)\times M_2 \subseteq \cdots \subseteq \varphi(M_i\times M)\times M_i \subseteq \cdots \subseteq A$$
of ideals in $A$. Therefore, $M_{i}=M_{i+1}$ for all $i\gg 0$ because $A$ is a Noetherian ring. Thus, $M$ is a Noetherian $R$-module. 

If $A$ is an Artinian ring, then so is $R$ since $\dim A=\dim R$.
\end{proof}


\begin{prop}\label{p34}
Let $R$ be a ring. Suppose that $M$ is a free $R$-module of rank $\ell>0$ with a free basis $\mathbf{e}_1, \dots, \mathbf{e}_\ell$. Then $\varphi$ has one of the following forms.
\begin{enumerate}[\rm(1)] 
\item If $\ell\ge 2$, then $\varphi=0$, i.e. $A$ is the idealization of $M$ in the sense of \cite{AW}.
\item If $\ell=1$, then $\varphi: R\times R \to R;\ (x \mathbf{e}_1,y\mathbf{e}_1)\mapsto axy$, where $a=\varphi(\mathbf{e}_1, \mathbf{e}_1)$ and $x,y \in R$. We then have $A\cong R[X]/(X^2-a)$ as rings.
\end{enumerate}
\end{prop}

\begin{proof}
Suppose that $\ell\ge 2$. For each $1\le i \le \ell$, there exists $1\le k_i \le \ell$ such that $k_i\ne i$. Hence, the definition of $\varphi$ requests that
\[
\varphi(\mathbf{e}_i, \mathbf{e}_j)\mathbf{e}_{k_i} = \varphi(\mathbf{e}_j, \mathbf{e}_{k_i})\mathbf{e}_i
\]
for all $1\le i,j \le \ell$. This implies that $\varphi(\mathbf{e}_i, \mathbf{e}_j)=0$ since $k_i\ne i$. It follows that $\varphi=0$.

Suppose that $\ell=1$ and set $a=\varphi(\mathbf{e}_1, \mathbf{e}_1)$. Then we have 
\[
\varphi(x\mathbf{e}_1, y\mathbf{e}_1)=xy\varphi(\mathbf{e}_1, \mathbf{e}_1)=axy
\]
for all $x,y \in R$. Let $R[X]$ be the polynomial ring over $R$. We have a surjective homomorphism of $R$-algebras $\psi: R[X] \to A; \ X\mapsto (0,\mathbf{e}_1)$. Then, it is straightforward to check that $\Ker \psi=(X^2-a)$.
\end{proof}

\begin{thm} \label{rank} 
Let $R$ be a Noetherian ring and $M$ a finitely generated $R$-module of rank $\ell>0$. Then the following hold true.
\begin{enumerate}[\rm(1)] 
\item If $\ell \ge 2$, then $\varphi=0$, i.e. $A$ is the idealization of $M$ in the sense of \cite{AW}.
\item Assume that $\ell=1$ and $M=I$ where $I$ is an ideal of $R$. Then 
\[
\varphi: I\times I \to R;\ \ (x,y)\mapsto \alpha xy
\] 
for some element $\alpha\in R:I^2$ $(\subseteq \rmQ(R))$. 
\end{enumerate}
\end{thm}

\begin{proof}
For the map $\varphi$, we consider the map
\[
\rmQ(R)\otimes_R \varphi: (\rmQ(R)\otimes_R M) \times (\rmQ(R)\otimes_R M) \to \rmQ(R); \ (x/s, y/t)\mapsto \varphi(x,y)/st,
\]
where $x,y\in M$ and $s,t\in R$ is a non-zerodivisor of $R$. Then one can check that $\rmQ(R)\otimes_R \varphi$ also satisfies the assertions to define the idealization of $\rmQ(R)\otimes_R M$ with respect to $\rmQ(R)\otimes_R \varphi$. Note that $\rmQ(R)\otimes_R M\cong \rmQ(R)^{\oplus \ell}$ since $M$ is of rank $\ell$.
If $\ell\ge 2$, then $\rmQ(R)\otimes_R \varphi=0$ by Proposition \ref{p34}. It follows that for all $x,y\in M$, $\varphi(x,y)/1=\rmQ(R)\otimes_R \varphi(x/1,y/1)=0$. Hence, we obtain that $\varphi(x,y)=0$, i.e. $\varphi=0$.

Suppose that $\ell=1$ and $M=I$ is an ideal of $R$. Note that $I$ contains a non-zerodivisor $w\in I$ of $R$ since $M=I$ is of rank $1$. Set $\alpha=\rmQ(R)\otimes_R \varphi(w/w,w/w)$. Then, for all $x,y\in I$, we obtain that 
\begin{align*}
\varphi(x,y)/1=\rmQ(R)\otimes_R \varphi(x/1,y/1)=\rmQ(R)\otimes_R \varphi(wx/w,wy/w)=xy\rmQ(R)\otimes_R \varphi(w/w,w/w)=\alpha xy.
\end{align*}
\end{proof}


The following is useful to verify the local condition of $A$.


\begin{prop}\label{local} 
We consider the following conditions.

\begin{enumerate}[$(1)$]
\item
$A$ is a local ring.
\item
$R$ is a local ring. 
\end{enumerate}
Then, the implication $(1)\Rightarrow(2)$ holds true, and the converse $(2)\Rightarrow(1)$ also holds if $\varphi(M\times M)$ does not contain any unit of $R$.

\end{prop}

\begin{proof}
$(1)\Rightarrow (2)$: Since $A$ is integral over $R$ by Remark \ref{integral}, if $A$ is a local ring, so is $R$.

$(2)\Rightarrow (1)$: Suppose that $(R, \m)$ is a local ring  and $\varphi(M\times M) \subseteq \m$. Then, $\m\times M$ is an ideal of $A$ by Proposition \ref{graded}. Let $(a, x)\in A\setminus (\m\times M)$. Since $a\in R$ is a unit of $R$, $a^2-\varphi(x, x)\in R$ is also a unit of $R$. Then, we have 
$$(a, x)\cdot \left( a(a^2-\varphi(x, x))^{-1}, -x (a^2-\varphi(x, x))^{-1}\right)=(1, 0)\ \text{in $A$},$$
which implies that $(a, x)$ is a unit of $A$. Thus, $(A, \m\times M)$ is a local ring, as desired.
\end{proof}

\begin{cor}\label{localc}
If $(R, \fkm)$ is a local ring and $\varphi(M\times M) \subseteq \fkm$, then $A$ is a local ring having the unique maximal ideal $\fkm \times M$.
\end{cor}

\begin{rem}
If $M$ has no free summands, then $\varphi(M\times M) \cap R^{\times}=\emptyset$. Indeed, if $\varphi(M\times M) \cap R^{\times}\ne \emptyset$, then there are elements $x,y\in M$ such that $\varphi(x,y)$ is a unit of $R$. Hence, an $R$-linear homomorphism $\varphi(x, -): M \to R; \ \ z\mapsto \varphi(x, z)$ is surjective and a split homomorphism. It follows that $M$ has a free summand. This is a contradiction.
\end{rem}

We cannot determine if $R\times_\varphi M$ is local or not when  $R$ is a local ring and $\varphi(M\times M) \cap R^{\times}\ne \emptyset$. 

\begin{ex}\label{ex48!}
Let $\mathbb{R}$ be the set of real numbers and $\mathbb{R}[X]$ denote the polynomial ring over $\mathbb{R}$. Then the following hold true.
\begin{enumerate}[\rm(1)] 
\item Let $A_1=\mathbb{R}\times_{\varphi_1} \mathbb{R}$, where $\varphi_1:\mathbb{R} \times \mathbb{R} \to \mathbb{R}; \ \ (x,y) \to xy$. Then $A_1\cong \mathbb{R}[X]/(X^2-1)=\mathbb{R}[X]/[(X+1)(X-1)]$ and thus $A$ is not a local ring.
\item Let $\mathbb{R}\times_{\varphi_2} \mathbb{R}$, where $\varphi_2:\mathbb{R} \times \mathbb{R} \to \mathbb{R}; \ \ (x,y) \mapsto -xy$. Then $A_2\cong \mathbb{R}[X]/(X^2+1)\cong \mathbb{C}$ and thus $A_2$ is a local ring.
\end{enumerate}

\end{ex}


Let $(R, \fkm)$ be a Noetherian local ring of dimension $d$, $M$ a nonzero finitely generated $R$-module, and $\varphi: M\times M \to R$ be an $R$-bilinear  homomorphism satisfying the following conditions.
\begin{enumerate}[\rm(1)] 
\item $\varphi(x, y)=\varphi(y,x)$ for all $x,y\in M$.
\item $\varphi(x,y)z=\varphi(y,z)x$ for all $x,y,z\in M$.
\item $\varphi(M\times M) \subseteq \fkm$.
\end{enumerate} 
 Set $A=R\times_\varphi M$ and $\fkn=\fkm \times M$.

The characterization of the Cohen-Macaulay property of $A$ follows from Proposition \ref{noether} and the following general fact.

\begin{fact}\label{f310}
Let $(R, \fkm)$ and $(S, \fkn)$ be Noetherian local rings. Let $\psi: R\to S$ be a homomorphism of rings and assume that $S$ is a finitely generated $R$-module. Then, $\depth S=\depth_R S$.
\end{fact}


By applying the above fact with $S=A=R\oplus M$ as $R$-modules, we have the following.

\begin{cor}\label{CM} 
The following are equivalent.
\begin{enumerate}[\rm(1)] 
\item $A$ is a Cohen-Macaulay ring.
\item $R$ is a Cohen-Macaulay ring and $M$ is a maximal Cohen-Macaulay $R$-module.
\end{enumerate}
\end{cor}

\begin{proof}
By Proposition \ref{noether}, in each implication, we may assume that $R$ is Noetherian and $M$ is a finitely generated $R$-module. By Fact \ref{f310}, We then have 
\[
\depth A=\depth_R A=\depth_R (R\oplus M) = \min\{\depth R, \depth_R M\}.
\]
This provides the assertion.
\end{proof}

\section{The almost Gorenstein property of $R\times_{\varphi} \m$}\label{s5}

In this section, we explore the almost Gorenstein property of $R\times_{\varphi} \m$ in dimension one. We begin with the following. 

\begin{prop}\label{p48}
Suppose that $R$ is a Cohen-Macaulay local ring having the canonical module $\rmK_R$ and  $M$ is a maximal Cohen-Macaulay $R$-module. Then, an $A$-module
\[
\Hom_R(M, \rmK_R)\oplus \rmK_R
\]
with the action of Proposition \ref{dualaction} is the canonical module of $A$.
\end{prop}

\begin{proof}
By \cite[Theorem 3.3.7(b)]{BH}, we have $\rmK_A\cong \Hom_R(A, \rmK_R)$. Hence, the assertion follows by Proposition \ref{p33}.
\end{proof}

\begin{lem}\label{l59}
Let $I$ be a fractional ideal of $R$. Then, $\rmr_R(I)=\ell_R((I:\fkm)/I)$.
\end{lem}

\begin{proof}
Consider the exact sequence $0 \to \fkm \to R \to R/\fkm \to 0$. By applying the functor $\Hom_R(-, I) = I:-$, we obtain that 
\[
0 \to I=I:R \to I:\fkm \to \Ext_R^1(R/\fkm, I) \to 0. 
\]
It follows $\Ext_R^1(R/\fkm, I) \cong (I:\fkm)/I$. 
\end{proof}

In what follows, let $(R, \m)$ be a Cohen-Macaulay local ring of dimension one having the canonical module $\rmK_R$, and suppose that there exists a canonical ideal. We set 
\begin{center}
$B:=\fkm:\fkm$, \quad  and $K:=\frac{\omega}{a}$
\end{center}
as in Section \ref{s1}.

\begin{lem}\label{l54}
Suppose that $R$ is not a discrete valuation ring. Then $\fkm(K:\fkm)=\fkm K$.
\end{lem}

\begin{proof}
Since $K\subseteq K:\fkm$, we have $\fkm(K:\fkm)\supseteq \fkm K$. We prove the reverse inclusion. In the equations below, $\rmr_R(\fkm K)=\mu_R(K:\fkm K)$ follows from \cite[Proposition 3.3.11]{BH}. Other equalities can be proved by Fact \ref{lll55}, Lemma \ref{l59}, and \cite[Lemma 2.1]{HK2}.
\begin{align*}
\ell_R((K:\fkm)/\fkm K)=&\ell_R((K:\fkm)/K) + \ell_R(K/\fkm K)=1 + \rmr(R) \quad \text{and}\\
\ell_R((\fkm K:\fkm)/\fkm K)=&\rmr_R(\fkm K)=\mu_R(K:\fkm K)=\mu_R((K:K):\fkm)=\mu_R(R:\fkm)\\
=&\ell_R((R:\fkm)/\fkm(R:\fkm))=\ell_R((R:\fkm)/\tr_R(\fkm))=\ell_R((R:\fkm)/\fkm)\\
=&\ell_R((R:\fkm)/R) + \ell_R(R/\fkm)=\rmr(R) +1.
\end{align*}
Hence, since $\fkm K:\fkm \subseteq K:\fkm$, we have $K:\fkm=\fkm K:\fkm$. 
It follows that $\fkm (K:\fkm) = \fkm (\fkm K:\fkm)\subseteq \fkm K$.
\end{proof}

Let $\varphi : \m\times \m \to R; (x, y)\mapsto \alpha xy$ where $\alpha \in \fkm:\fkm^2 (\subseteq \rmQ(R))$. Set 
\begin{center}
$A=R\times_\varphi \fkm$ \quad and \quad $\fkn=\fkm \times \fkm$.
\end{center}
$A$ is a local ring with maximal ideal $\fkn$ by Corollary \ref{localc}. When the case where $\alpha=0$, that is, $A$ is the idealization in the sense of \cite{AW}, it is known that $A$ is almost Gorenstein if and only if $R$ is almost Gorenstein (\cite[Theorem 6.5]{GMP}).
The purpose of this section is to complement this result for any $\alpha\in \rmQ(R)$. 
The goal is to prove Theorem \ref{t55}. 

\begin{lem}\label{lll58}
$\rmQ(A)=\rmQ(R)\times \rmQ(R)$, where the product is given by 
\[
(a,x){\cdot}(b,y):=(ab+\alpha xy, ay+bx),
\]
for $a,b, x, y\in \rmQ(R)$.
\end{lem}

\begin{proof}
Let $s\in \fkm$ be a non-zerodivisor of $R$, and set the multiplicatively closed subset $S=\{s^n \mid n\ge 0\}$ of $R$. Then $R\subseteq S^{-1}R \subseteq \rmQ(R)$. Since $S^{-1}R$ is Artinian, it follows that $\rmQ(R)=\rmQ(S^{-1}R)=S^{-1}R$.
By the same reason, we have $\rmQ(A)=S^{-1} A$. Note that $S^{-1} A = S^{-1}R \times S^{-1}\fkm=\rmQ(R) \times \rmQ(R)$ as $R$-modules. The equations also hold as $A$-modules with the above product.
\end{proof}


Now, we are in a position to give a characterization of the almost Gorenstein property of $R\times_{\varphi} \m$.

\begin{thm}\label{t55}
Suppose that $R$ is not a discrete valuation ring. Consider the following conditions.
\begin{enumerate}[\rm(1)] 
\item $A$ is an almost Gorenstein ring.
\item $R$ is an almost Gorenstein ring and $\tr_B(\langle 1, \alpha \rangle_B)=B$, where $\langle 1, \alpha \rangle_B$ denotes the $B$-module generated by $1$ and $\alpha$.
\item $R$ is an almost Gorenstein ring and either $\alpha\in B$ or $\alpha^{-1}\in B$.
\end{enumerate}
Then, {\rm (1)$\Leftrightarrow$(2)$\Leftarrow$(3)} hold. {\rm (2)$\Rightarrow$(3)} also holds if $B$ is a local ring.
\end{thm}

\begin{proof}
First, we prove the following.

\begin{claim}
$\tr_A(\fkn \rmK_A)=\fkm (\tr_B(\langle 1, \alpha \rangle_B K) )\times \fkm (\tr_B(\langle 1, \alpha \rangle_B K)).$
\end{claim}

\begin{proof}
 By Proposition \ref{p48} and Lemma \ref{lll58}, a canonical module of $A$ is given by 
\[
\rmK_A=(K:\fkm) \times K \subseteq \rmQ(R)\times \rmQ(R)=\rmQ(A).
\]
Hence, 
\begin{align*}
\fkn \rmK_A&=[\fkm (K:\fkm) + \alpha \fkm K] \times [\fkm(K:\fkm) + \fkm K]\\
&= (\fkm K + \alpha \fkm K) \times \fkm K\\
&=\fkm \langle 1, \alpha \rangle_B K \times \fkm K
\end{align*}
by Lemma \ref{l54}. In contrast, for $(a, x)\in \rmQ(A)$, we can confirm that
\begin{eqnarray*}
&& (a, x)\in A:\fkn \rmK_A \\
 &\Leftrightarrow& (a, x)\cdot [\langle 1, \alpha \rangle_B\fkm K \times \fkm K] \subseteq R\times \fkm \\
&\Leftrightarrow& a\in [R: \langle 1, \alpha \rangle_B \fkm K]\cap (\fkm :\fkm K) \text{ and } x\in \fkm: \langle 1, \alpha \rangle_B \fkm K.
\end{eqnarray*}
Since $R: \langle 1, \alpha \rangle_B \fkm K=(R: \fkm):\langle 1, \alpha \rangle_B  K =(\fkm: \fkm):\langle 1, \alpha \rangle_B  K=B:\langle 1, \alpha \rangle_B  K$ and $\fkm: \langle 1, \alpha \rangle_B \fkm K=(\fkm: \fkm):\langle 1, \alpha \rangle_B  K=B:\langle 1, \alpha \rangle_B  K$, we have $A: \fkn \rmK_A=(B:\langle 1, \alpha \rangle_B  K)\times (B:\langle 1, \alpha \rangle_B  K)$.
Therefore, we obtain that
\begin{align*}
\tr_A(\fkn \rmK_A)&=\fkn \rmK_A(A:\fkn \rmK_A)\\
&= (\fkm \langle 1, \alpha \rangle_B K \times \fkm K)\cdot [(B:\langle 1, \alpha \rangle_B  K)\times (B:\langle 1, \alpha \rangle_B  K)]\\
&=[\fkm \langle 1, \alpha \rangle_B K (B:\langle 1, \alpha \rangle_B  K) + \alpha \fkm K (B:\langle 1, \alpha \rangle_B  K)]\times[\fkm \langle 1, \alpha \rangle_B K (B:\langle 1, \alpha \rangle_B  K) + \fkm K(B:\langle 1, \alpha \rangle_B  K)]\\
&=\fkm \langle 1, \alpha \rangle_B K (B:\langle 1, \alpha \rangle_B  K) \times \fkm \langle 1, \alpha \rangle_B K (B:\langle 1, \alpha \rangle_B  K) \\
&=\fkm (\tr_B(\langle 1, \alpha \rangle_B K)) \times \fkm (\tr_B(\langle 1, \alpha \rangle_B K)).
\end{align*}

\end{proof}

(1)$\Leftrightarrow$(2): Since $\fkm (\tr_B(\langle 1, \alpha \rangle_B K)) \subseteq \fkm B =\fkm$, we have $\tr_A(\fkn \rmK_A)\subseteq \fkm \times \fkm = \fkn$. Hence, $A$ is an almost Gorenstein ring if and only if $\fkm (\tr_B(\langle 1, \alpha \rangle_B K))=\fkm$ by Theorem \ref{p52}. 

In contrast, since $\tr_B(\langle 1, \alpha \rangle_B K)=[\langle 1, \alpha \rangle_B  (B:\langle 1, \alpha \rangle_B  K)]K$, by applying Theorem \ref{p52} with $X=\langle 1, \alpha \rangle_B  (B:\langle 1, \alpha \rangle_B  K)$, we obtain that $R$ is an almost Gorenstein ring and $\fkm K= \fkm$ if $\fkm (\tr_B(\langle 1, \alpha \rangle_B K))=\fkm$.
In addition, if $\fkm K= \fkm$, we have 
\begin{align*}
\fkm (\tr_B(\langle 1, \alpha \rangle_B K))&=\fkm \langle 1, \alpha \rangle_B K (B:\langle 1, \alpha \rangle_B  K)\\
&=\fkm K \langle 1, \alpha \rangle_B  [(\fkm:\fkm):\langle 1, \alpha \rangle_B  K]\\
&=\fkm K \langle 1, \alpha \rangle_B  (\fkm:\langle 1, \alpha \rangle_B  \fkm K)\\
&=\fkm  \langle 1, \alpha \rangle_B  (\fkm:\langle 1, \alpha \rangle_B  \fkm )\\
&=\fkm  \langle 1, \alpha \rangle_B  (B:\langle 1, \alpha \rangle_B )\\
&=\fkm (\tr_B(\langle 1, \alpha \rangle_B )).
\end{align*}
Thus, $\fkm (\tr_B(\langle 1, \alpha \rangle_B K))=\fkm$ is equivalent to saying that $R$ is almost Gorenstein and $\fkm (\tr_B(\langle 1, \alpha \rangle_B ))=\fkm$.
By Lemma \ref{ll55}, we can replace the equation with the equation in the assertion (2).

(3)$\Rightarrow$(2): Suppose that $\alpha\in B$. Then, we obtain that $\tr_B(\langle 1, \alpha \rangle_B )=\tr_B(B)=B$.
Suppose that $\alpha^{-1}\in B$. Then, since $\langle 1, \alpha \rangle_B\cong \alpha^{-1} \langle 1, \alpha \rangle_B= \langle 1, \alpha^{-1} \rangle_B= B$, it follows by a similar argument that $\tr_B(\langle 1, \alpha \rangle_B )=B$.

(2)$\Rightarrow$(3): Since $\tr_B(\langle 1, \alpha \rangle_B)=B$ and $B$ is a local ring, $\langle 1, \alpha \rangle_B$ is a cyclic $B$-module by \cite[Proposition 2.8 (iii)]{Lin}.
Therefore, we have either $\langle 1, \alpha \rangle_B= B$ or $\langle 1, \alpha \rangle_B=\alpha B$, which implies that either $\alpha \in B$ or $\alpha^{-1}\in B$.
\end{proof}


\begin{ex}
Let $k$ be a field, and let $k[[t]]$ be the formal power series ring over $k$. 
\begin{enumerate}[\rm(1)] 
\item Set $R=k[[t^3, t^4, t^5]]$. Let $\varphi : \m\times \m \to R; (x, y)\mapsto \alpha xy$, where $\fkm$ denotes the maximal ideal of $R$ and $\alpha\in \fkm:\fkm^2$. Then, $R\times_{\varphi} \fkm$ is a non-Gorenstein almost Gorenstein ring for each $\alpha\in \fkm:\fkm^2$.
\item Set $R=k[[t^4, t^7, t^9]]$. Let $\varphi : \m\times \m \to R; (x, y)\mapsto \alpha xy$, where $\fkm$ denotes the maximal ideal of $R$ and $\alpha\in \fkm:\fkm^2$. Then, $R\times_{\varphi} \fkm$ is a non-Gorenstein almost Gorenstein ring for each $\alpha\in \fkm:\fkm$, but not an almost Gorenstein ring for each $\alpha\in (\fkm:\fkm^2)\setminus (\fkm:\fkm)$.
\end{enumerate}
\end{ex}

\begin{proof}
(1): We note that $R$ is almost Gorenstein (\cite[Example 3.2(1)]{GMP}) and that $\fkm:\fkm^2=\langle t^{-3}, t^{-2}, t^{-1}\rangle_R$, $\fkm:\fkm=k[[t]]$. We denote by 
\[
v\colon \rmQ(R)\to\mathbb{Z}\cup\{\infty\}
\] 
the {\it normalized valuation} associated to $k[[t]]$.  Let $\alpha \in \fkm:\fkm^2=\langle t^{-3}, t^{-2}, t^{-1}\rangle_R$. If $v(\alpha)\ge 0$, then $\alpha\in k[[t]]=\fkm:\fkm$. 
If $-3\le v(\alpha)< 0$, then $0< v(\alpha^{-1})\le 3$; hence, $\alpha^{-1}\in k[[t]]=\fkm:\fkm$. Hence, $R\times_{\varphi} \fkm$ is almost Gorenstein by Theorem \ref{t55}. By Corollary \ref{c417}, $R$ is not Gorenstein.

(2): We note that $R$ is almost Gorenstein (\cite[Example 4.3(2)]{GMP}) and that $\fkm:\fkm^2=\langle 1, t, t^3, t^6\rangle_R$, $\fkm:\fkm=k[[t^4, t^5, t^7]]$. Let $\alpha \in \fkm:\fkm^2=\langle 1, t, t^3, t^6\rangle_R = k+kt+\sum_{n\ge 3} kt^n$, and write 
\[
\alpha=a_0+a_1 t +\sum_{n\ge 3} a_n t^n 
\]
for $a_i\in k$. Suppose that $\alpha^{-1}\in B$, equivalently, $1\in \alpha B$. This is also equivalent to saying that there exists $\beta\in B$ such that $\alpha \beta =1$. Write 
\[
\beta=b_0 +b_4t^4+b_5t^5+\sum_{m\ge 7}b_m t^m
\]
for $b_j\in k$. Then, 
\begin{align*}
1=&\alpha \beta \\
=& a_0b_0+a_1b_0 t+a_3b_0 t^3+(a_0b_4+a_4b_0) t^4 +(a_0b_5+a_1b_4+a_5b_0)t^5 +(a_1b_5+a_6b_0)t^6 +(\text{higher terms})
\end{align*}
Since $a_0, b_0\ne 0$, we observe that $a_1=a_3=0$. It also follows that $a_6=0$. Therefore, $\alpha=a_0 +a_4t^4+a_5t^5+\sum_{n\ge 7}a_n t^n\in B$. Thus, $\alpha^{-1}\in B$ implies that $\alpha\in B$. Therefore, by Theorem \ref{t55}, $R$ is almost Gorenstein if and only if $\alpha\in B$. On the other hand, by Corollary \ref{c417}, $R$ is not Gorenstein.
\end{proof}

\section{Appendix 1: The Gorenstein property of $\mathbb{Z}_2$-graded local rings}\label{s4}
The Gorenstein property of $\mathbb{Z}_2$-graded local rings is already well established (see, for example, \cite{C, F, R}).
Nevertheless, as we have found an alternative and insightful proof concerning the Gorenstein property of  $A=R\times_\varphi M$, we present it in the appendix.

In what follows, we assume the following setup. 

\begin{setup}\label{setup2}
Let $(R, \fkm)$ be a Noetherian local ring of dimension $d$, $M$ a nonzero finitely generated $R$-module, and $\varphi: M\times M \to R$ be an $R$-bilinear  homomorphism satisfying the following conditions.
\begin{enumerate}[\rm(1)] 
\item $\varphi(x, y)=\varphi(y,x)$ for all $x,y\in M$.
\item $\varphi(x,y)z=\varphi(y,z)x$ for all $x,y,z\in M$.
\item $\varphi(M\times M) \subseteq \fkm$.
\end{enumerate} 
 Set $A=R\times_\varphi M$ and $\fkn=\fkm \times M$.
\end{setup}

\subsection{The case of Artinian rings}\label{subsection4.1}



We first explore the case of Artinian rings. In addition to Setup \ref{setup2}, suppose that $d=0$. We set 
\[
M_{\varphi}=\{ x \in M \mid \text{$\varphi(x,y)=0$ for all $y\in M$}\}
\]
to state our assertions simply. Recall that $M$ is called {\it faithful} if $\Ann_R M=0$.


\begin{prop}\label{CMtype} 
The equality $(0):_A \fkn = [((0):_R \fkm) \cap \Ann_R M] \times [((0):_{M} \fkm)\cap M_\varphi]$ holds.
\end{prop}

\begin{proof}
Let $(a, x)\in (0):_A \n$. Then, we have
$$(a, x)\cdot (b, 0)=(ab, bx)=(0, 0)$$
 for every $b\in \m$, which implies that $a\in (0):_R \m$ and $x\in (0):_M \m$. Similarly, we have
 $$(a, x)\cdot (0, y)=(\varphi(x, y), ay)=(0, 0)$$
for every $y\in M$, which implies that  $a\in \Ann_R M$ and $x\in M_{\varphi}$. 

Conversely, let $a\in ((0):_R \fkm) \cap \Ann_R M$ and $x\in ((0):_M \fkm) \cap M_{\varphi}$. Then
$$(a, x)\cdot (b, y)=(ab+\varphi(x, y), bx+ay)=(0, 0)$$
for every $(b, y)\in \m\times M=\n$, which induces $(a, x)\in (0):_A\n$.
\end{proof}


We prepare a lemma on $M _\varphi$ to characterize the Gorenstein property of $A$.

\begin{lem}\label{l43}
We have $M_\varphi  \subseteq  (0):_M \varphi(M\times M)$. Furthermore, if $M$ is a faithful $R$-module, then $M_\varphi  =  (0):_M \varphi(M\times M)$.
\end{lem}

\begin{proof}
Let $x\in M_\varphi$. Since $\varphi(x, y)=0$ for any $y\in M$, we have 
$$\varphi(y, z)x=\varphi(x, y)z=0$$
for any $z\in M$, which implies that $x\in (0):_M \varphi(M\times M)$. 

Suppse that $M$ is a faithful $R$-module. Let $x\in (0):_M \varphi(M\times M)$. Then, we have 
$$\varphi(x, y)z=\varphi(y, z)x=0$$
for any $y, z\in M$, which implies that $\varphi(x, y)M=0$. Since $M$ is faithful, we obtain that $\varphi(x, y)=0$ for all $y\in M$.
\end{proof}

The following is a characterization of the Gorenstein property of $A$.

\begin{thm} \label{Gor} 
The following are equivalent.
\begin{enumerate}[\rm(1)] 
\item $A$ is a Gorenstein ring $($i.e. $A$ is self-injective$)$.
\item Either of the following holds:
\begin{enumerate}[\rm(i)] 
\item $M$ is isomorphic to the canonical $R$-module.
\item $R$ is Gorenstein and $M_{\varphi}=0$.
\end{enumerate}
\end{enumerate}
\end{thm}

\begin{proof}

$(2)\Rightarrow(1)$: Firstly, suppose that $M$ is isomorphic to the canonical module of $R$. Then, $\Ann_R M=0$ and $\ell_R((0):_M \m)=1$. By Proposition \ref{CMtype}, 
$$0<\ell_A((0):_A \n)=\ell_A(0\times [((0):_M \fkm) \cap M_{\varphi}] )\le \ell_A(0\times [(0):_M \fkm] )=\ell_R((0):_M \fkm)=1,$$ 
where the fourth equality follows from $R/\fkm \cong A/\fkn$. This  implies that $A$ is a Gorenstein ring.

Secondly, suppose that $R$ is a Gorenstein ring and $M_\varphi =0$.
Then, $\ell_R((0):_R \m) =1$ and we have 
$$0<\ell_A((0):_A \n)=\ell_A([((0):_R \fkm) \cap \Ann_R M]\times 0)\le \ell_A([(0):_R \fkm] \times 0)=\ell_R((0):_R \m)=1.$$
Therefore, $A$ is a Gorenstein ring.

$(1)\Rightarrow(2)$: Suppose that $A$ is a Gorenstein ring.
Firstly, suppose that $M$ is faithful. By Lemma \ref{l43}, 
$$((0):_{M} \fkm)\cap M_\varphi=((0):_M \fkm) \cap ((0):_M \varphi(M\times M))=(0):_M \m.$$
It follows that
$$1=\ell_A((0):_A \fkn) =\ell_A(0 \times [(0):_M \m])=\ell_R((0):_M\m),$$
which induces that $M$ is isomorphic to the canonical module of $R$ (\cite[Proposition 3.3.13]{BH}).

Secondly, suppose that $M$ is not faithful. Then, 
\begin{align}\label{eqgor}
1=\ell_A((0):_A \fkn)=\ell_R(((0):_R \fkm) \cap \Ann_R M)+\ell_R(((0):_{M} \fkm)\cap M_\varphi)
\end{align}
and $((0):_R \fkm) \cap \Ann_R M\neq 0$. Hence, $((0):_{M} \fkm)\cap M_\varphi=0$. It follows that $M_\varphi=0$.
On the other hand, note that
$$\varphi(((0):_R\m)M\times M)=((0):_R \m)\cdot \varphi(M\times M)=0$$
by the assumption that $\varphi(M\times M) \subseteq \fkm$. Hence, $((0):_R\m)M\subseteq M_\varphi=0$. Therefore, we obtain that $(0):_R\m\subseteq \Ann_R M$, and which induces that $R$ is a Gorenstein ring by \eqref{eqgor}. 
\end{proof}

We should emphasis that the condition (2)(ii) of Theorem \ref{Gor} does not follow from the following Fossum's result \cite[Theorem]{F}, and does not appear in the characterization of the Gorenstein property of idealizations (\cite[(7) Theorem]{R}).

\begin{fact}
Let $S$ be a Noetherian local ring. Suppose that $R$ is a Noetherian local ring and $\rmK_R$ is the canonical module of $R$. If $0 \to \rmK_R \to S \to R \to 0$ is a commutative extension of $A$ by $M$, then $S$ is Gorenstein.
\end{fact}


We explore the condition (2)(ii) of Theorem \ref{Gor} in details.



\begin{prop}\label{l42}
Choose elements $x_1,\dots, x_s$ and $y_1, \dots, y_t$ such that
\begin{align*}
M=\langle x_1, x_2, \dots, x_s\rangle \quad \text{and} \quad  
(0):_M \fkm=\langle y_1, y_2, \dots, y_t \rangle,
\end{align*}
where $s$ is  the number of minimal generators $\mu_R(M)=\ell_R(M/\fkm M)$ of $M$ and $t$ is the Cohen-Macaulay type $\rmr_R(M)=\ell_R((0):_M \fkm)$ of $M$.
Then the following are equivalent.
\begin{enumerate}[\rm(1)] 
\item $M_\varphi=0$.
\item The product of the matrices
\[
\begin{pmatrix}
&&\\
&\varphi(x_i, y_j)&\\
&&
\end{pmatrix}
\begin{pmatrix}
c_1 \\
c_2\\
\vdots \\
c_t
\end{pmatrix}
\]
is nonzero for all unit elements $c_1,c_2, \dots, c_t\in R^\times$.
\end{enumerate}
\end{prop}

\begin{proof}
Let $y\in (0):_M \m$ and write $y=c_1y_1+c_2y_2+\cdots +c_ty_t$ with $c_1, c_2, \dots, c_t\in R^\times \cup \{0\} $.
Then, we can confirm that
\begin{eqnarray*}
&& y\in  ((0):_{M} \m) \cap M_\varphi \\
 &\Leftrightarrow& \varphi(x, y)=0 \ \text{for all}\ x\in M\\
&\Leftrightarrow& \sum_{j=1}^t c_j\varphi(x, y_j)=0 \ \text{for all}\ x\in M\\
&\Leftrightarrow& \sum_{j=1}^t c_j\varphi(x_i, y_j)=0\ \text{for all}\ 1\le i\le s\\
&\Leftrightarrow&
\begin{pmatrix}
&&\\
&\varphi(x_i, y_j)&\\
&&
\end{pmatrix}
\begin{pmatrix}
c_1 \\
c_2\\
\vdots \\
c_t
\end{pmatrix}
=\mathbf{0}.
\end{eqnarray*}

Since $M_{\varphi}=0$ if and only if $((0):_{M}\m) \cap M_{\varphi}=0$, we obtain the equivalence $(1)\Leftrightarrow (2)$. 
\end{proof}

Let $(R, \m)$ be an Artinian Gorenstein local ring with the residue field $K$, and $M=K^{\oplus s}$ for some $s>0$. Let $\varphi: M\times M \to R$ be an  $R$-bilinear  homomorphism satisfying  conditions of Definition \ref{def1}.
Then, $\varphi(M\times M)\subseteq (0):_R\m$. Since $R$ is Gorenstein, there exists a natural isomorphism $\iota: (0):_R\m \to K$. 
With the notation, we obtain the following.

\begin{thm} \label{Gormatrix}
Suppose that $(R, \fkm)$ is an Artinian Gorenstein local ring containing a field $K$ such that a canonical homomorphism $K \to R \to R/\fkm$ of rings is bijective. Suppose that $M=K^{\oplus s}$ for some $s>0$. Suppose that $R$ is not a field.
Choose a $K$-basis $\mathbf{e}_1, \mathbf{e}_2, \dots, \mathbf{e}_s$ of $M$. Set 
\[
\Phi=\left\{ \iota \circ \varphi:M\times M \overset{\varphi}{\to} (0):_R \fkm \cong K \ \middle| \ 
\begin{matrix}
\text{$\varphi$ is an $R$-bilinear homomorphism such that } \\
\text{$\varphi(x, y)=\varphi(y,x)$ and $\varphi(x,y)z=\varphi(y,z)x$ for all $x,y,z\in M$}
\end{matrix}\right\}
\]
and 
\[
S=\left\{ C=(c_{ij})\in \rmM(s, K) \mid c_{ij}=c_{ji} \text{ for all $1\le i\le j \le \ell$} \right\},
\]
where $\rmM(s, K)$ denotes the set of $s\times s$ matrices whose entries are in $K$. 
Then
\[
f: \Phi \to S; \ \ \iota \circ \varphi \mapsto C_\varphi:= ((\iota\circ\varphi)(\mathbf{e}_i, \mathbf{e}_j))
\] 
is a one-to-one correspondence. With the correspondence, we obtain that $A=R\times_\varphi M$ is Gorenstein if and only if $\det C_\varphi\ne 0$. 
\end{thm}

\begin{proof}
Let $\iota \circ \varphi\in \Phi$. Since $\varphi(\mathbf{e}_i, \mathbf{e}_j)=\varphi(\mathbf{e}_j, \mathbf{e}_i)$ for all $1\le i\le j \le \ell$, the map $f$ is well-defined.
On the other hand, for $C=(c_{ij})\in S$, we define the map 
$$\varphi_{C}: M\times M\to (0):_R\m,\  (\mathbf{e}_i, \mathbf{e}_j)\mapsto c_{ij}\cdot \xi,$$
where $\xi\in R$ is a $K$-base of $(0):_R \m$. 
 Then, $\varphi_C$ is an $R$-bilinear homomorphism such that $\varphi_C(x, y)=\varphi_C(y,x)$ for all $x,y\in M$ because $c_{ij}=c_{ji}$. 
By recalling that $M$ is a $K$-vector space, $\varphi_C(M\times M)\subseteq (0):_R\m$. Hence, 
 $$\varphi(x, y)z=0=\varphi(y, z)x$$ for any $x, y, z\in M$. Therefore, the map
 $$g: S\to \Phi; \ \ C\mapsto \iota \circ \varphi_C$$
 is well-defined. It is easy to check that $f\circ g=id_S$ and $g\circ f=id_{\Phi}$, hence $f$ is bijective.

In contrast, by Proposition \ref{l42}, $M_\varphi=0$ if and only if the equation $C_{\varphi}\mathbf{x}=\mathbf{0}$ has only trivial solutions, which is equivalent to $\det C_\varphi \neq 0$. Therefore, we obtain that $A=R\times_{\varphi} M$ is a Gorenstein ring if and only if $\det C_\varphi\neq 0$ by Theorem \ref{Gor}, as desired (notice that $A$ is also a local ring since $\varphi(M\times M)\subseteq (0):_R \m$).
\end{proof}

Furthermore, we can determine the defining ideal of $R\times_{\varphi_C} M$.

\begin{prop}\label{p46}
Suppose that $(R, \fkm)$ is an Artinian Gorenstein local ring 
containing a field $K$ such that a canonical homomorphism $K \to R \to R/\fkm$ of rings is bijective. Suppose that $R\ne K$, and choose $\xi\in \fkm$ such that $R\xi=(0):_R \m$.  Suppose that $M=K^{\oplus s}$ for some $s>0$. 
Choose a $K$-basis $\mathbf{e}_1, \mathbf{e}_2, \dots, \mathbf{e}_s$ of $M$. Let $C=(c_{ij})$ be a symmetric matrix  whose entries $c_{ij}$ are in $K$. Let $\varphi_C$ denotes the defining map of $R\times_{\varphi_C} M$, that is, 
$$\varphi_{C}: M\times M\to R;\  (\mathbf{e}_i, \mathbf{e}_j)\mapsto c_{ij}\cdot \xi.$$
Let 
\[
f: R[X_1, X_2, \dots, X_s] \to R\times_{\varphi_C} M; \ X_i \mapsto (0, \mathbf{e}_i)
\]
be an $R$-algebra homomorphism, where $R[X_1, X_2, \dots, X_s]$ denotes the polynomial ring over $R$. Then, 
\[
\Ker f=(X_iX_j - c_{ij}\xi \mid 1\le i,j\le s) + \fkm(X_1, X_2, \dots,X_s).
\]
\end{prop}

\begin{proof}
Set $S:=R[X_1, X_2, \dots, X_s]$ and $I:=(X_iX_j - c_{ij}\xi \mid 1\le i,j\le s) + \fkm(X_1, X_2, \dots,X_s)$. Consider an exact sequence
\[
0 \to (R+I)/I \to S/I \to S/(R + I) \to 0
\]
of $R$-modules. Then, there exists a canonical surjective homomorphism $R \to R/(I\cap R) \cong (R+I)/I$ and an isomorphism $S/(R + I) =S/[R + (X_iX_j  \mid 1\le i,j\le s) + \fkm(X_1, X_2, \dots,X_s)] \cong \sum_{i=1}^s KX_i$. Hence, 
\[
\ell_R(S/I) = \ell_R((R+I)/I) + \ell_R(S/(R + I)) \le \ell_R(R) + s.
\] 
On the other hand, since $I\subseteq \Ker f$, we have $\ell_R(S/\Ker f) \le \ell_R(S/I)$. In addition, $\ell_R(S/\Ker f) = \ell_R(R\times_{\varphi_C} M) = \ell_R(R) + \ell_R(M) = \ell_R(R) + s$. It follows that $\Ker f=I$.
\end{proof}

By combining Theorem \ref{Gormatrix} and Proposition \ref{p46}, we obtain a family of certain Artinian Gorenstein rings.


\begin{cor}\label{c47}
Let $s>0$. Suppose that $(R, \fkm)$ is an Artinian Gorenstein local ring containing a field $K$ such that a canonical homomorphism $K \to R \to R/\fkm$ of rings is bijective. Suppose that $R\ne K$, and choose $\xi\in \fkm$ such that $R\xi=(0):_R \m$. Let $C=(c_{ij})$ be a symmetric matrix  whose entries are in $K$. Set 
\[
A_C=R[X_1, X_2, \dots, X_s]/[(X_iX_j - c_{ij}\xi \mid 1\le i,j\le s) + \fkm(X_1, X_2, \dots,X_s)].
\]
Then, $A_C$ is Gorenstein if and only if $\det C\ne 0$.
\end{cor}

\begin{ex}\label{aaa48}
Let $K$ be a field and $K[X, Y, Z]$ be the polynomial ring over $K$. For elements $a,b, c\in K$, set 
\[
A_{a,b,c}:=K[X, Y, Z]/(X^2, XY, XZ, Y^2-aX, YZ-bX, Z^2-cX).
\]
Then, $A_{a,b,c}$ is Gorenstein if and only if $ac\ne b^2$.
\end{ex}

\begin{proof}
We apply Corollary \ref{c47} with $A=K[X]/(X^2)$ and $s=2$.
\end{proof}


\subsection{The case of higher-dimensional rings}\label{subsection42}
In this subsection we explore the Gorenstein property of $A$ in arbitrary dimension. We maintain Setup \ref{setup2}.

Since an $A$-module $\Hom_R(M, \rmK_R)\oplus \rmK_R$ is the canonical module of $A$ (Proposition \ref{p48}), we can calculate the Cohen-Macaulay type of $A$. To state our assertion simply, set 
\[
\psi_{M, \rmK_R}=\{ \psi_{x, k}\in \Hom_R(M, \rmK_R) \mid x\in M, k\in \rmK_R\}.
\]
We also recall the notion of trace modules. 

\begin{defn} {\rm (\cite[Definition 2.1]{LP})}\label{d410}
For $R$-modules $M$ and $N$, 
\begin{align*} 
\mathrm{tr}_{M}(N):=\sum_{f\in \Hom_R(N, M)} \Im f \ \subseteq M
\end{align*}
is called the {\it trace module} of $N$ in $M$. 
\end{defn}

\begin{prop}\label{CMtype2}
Suppose that $(R, \fkm)$ is a Cohen-Macaulay local ring having the canonical module $\rmK_R$ and  $M$ is a maximal Cohen-Macaulay $R$-module. Then, 
\[
\rmr(A) = \ell_R(\Hom_R(M, \rmK_R)/[\fkm \Hom_R(M, \rmK_R) + \psi_{M, \rmK_R}]) + \ell_R(\rmK_R/[\fkm \rmK_R + \tr_{\rmK_R}(M)]).
\]
\end{prop}

\begin{proof}
Since $\rmr(A)=\mu_A(\rmK_A)$, the assertion can be checked by calculating 
\[
\ell_A([\Hom_R(M, \rmK_R)\oplus \rmK_R]/\fkn[\Hom_R(M, \rmK_R)\oplus \rmK_R])
\]
(recall the definition of the $A$-action of $\Hom_R(M, \rmK_R)\oplus \rmK_R$ given in Proposition \ref{dualaction}).
\end{proof}

We recall the notion of residually faithful modules, which is introduced by Brennan and Vasconcelos \cite{BV}. We use several fundamental facts on residually faithful modules to obtain a characterization of the Gorenstein property of $A$.

\begin{defn}{\rm (\cite[Definition 5.1]{BV})}
Let $N$ be a maximal Cohen-Macaulay $R$-module. We say that $N$ is a {\it residually faithful $R$-module} if $N/\fkq N$ is a faithful $R/\fkq$-module for some parameter ideal $\fkq$ of $R$.
\end{defn}

\begin{fact}{\rm (\cite[Proposition 3.2, Corollary 3.4, Proposition 3.6]{GKL})}\label{f311}
Suppose that $(R, \fkm)$ is a Cohen-Macaulay local ring having the canonical module $\rmK_R$. Let $N$ be a maximal Cohen-Macaulay $R$-module. The following are equivalent.
\begin{enumerate}[\rm(1)] 
\item $N$ is a residually faithful $R$-module.
\item $\tr_{\rmK_R}(N) = \rmK_R$.
\item $N/\fkq N$ is a faithful $R/\fkq$-module for all parameter ideals $\fkq$ of $R$.
\end{enumerate}
\end{fact}


\begin{thm}\label{higherGor}
Suppose that $R$ is a homomorphic image of a Gorenstein local ring. The following are equivalent.
\begin{enumerate}[\rm(1)] 
\item $A$ is Gorenstein.
\item One of the following hold. 
\begin{enumerate}[\rm(i)] 
\item $R$ is a Cohen-Macaulay ring having the canonical module $\rmK_R$ and $M\cong \rmK_R$.
\item $R$ is a Gorenstein ring, $M$ is a maximal Cohen-Macaulay $R$-module, and $\Hom_R(M, R) = \psi_{M, R}$.
\end{enumerate}
\end{enumerate}
\end{thm}

\begin{proof}
In the proof of each implication, we may assume that $R$ is a Cohen-Macaulay ring and that $M$ is a maximal Cohen-Macaulay $R$-module by Theorem \ref{CM}. Since $R$ is a homomorphic image of a Gorenstein local ring, $R$ has the canonical module $\rmK_R$ (\cite[Theorem 3.3.6]{BH}). 
By Proposition \ref{CMtype2}, $A$ is Gorenstein if and only if either 
\begin{enumerate}[\rm(a)] 
\item $\ell_R(\Hom_R(M, \rmK_R)/[\fkm \Hom_R(M, \rmK_R)+\psi_{M, \rmK_R}])=1$ and $\ell_R(\rmK_R/[\fkm \rmK_R + \tr_{\rmK_R}(M)])=0$ or
\item $\ell_R(\Hom_R(M, \rmK_R)/[\fkm \Hom_R(M, \rmK_R)+\psi_{M, \rmK_R}])=0$ and $\ell_R(\rmK_R/[\fkm \rmK_R + \tr_{\rmK_R}(M)])=1$.
\end{enumerate}
Hence, we only need to prove the following claim.
\end{proof}

\begin{claim}\label{claim1}
\begin{enumerate}[\rm(1)] 
\item The condition {\rm (a)} holds if and only if $M\cong \rmK_R$.
\item The condition {\rm (b)} holds if and only if $R$ is a Gorenstein ring and $\Hom_R(M, R) = \psi_{M, R}$.
\end{enumerate}
\end{claim}

\begin{proof}[Proof of Claim \ref{claim1}]
(1) (if part): If $M\cong \rmK_R$, then $\tr_{\rmK_R}(M)=\rmK_R$. Hence, $\ell_R(\rmK_R/[\fkm \rmK_R + \tr_{\rmK_R}(M)])=0$. Let $x\in M$ and $k\in \rmK_R$. Then, $\Im \psi_{x, k} = \{\varphi(x, y)k \mid y\in M\} \subseteq \fkm \rmK_R$ since $\varphi(M\times M) \subseteq \fkm$. By noting that $\Hom_R(\rmK_R, \rmK_R)\cong R$, it follows that $\psi_{x, k}$ is a homomorphism obtained by a multiplication of some $a\in \fkm$. Hence, $\psi_{M, \rmK_R}\subseteq \fkm \Hom_R(M, \rmK_R)$. Therefore, since $\Hom_R(M, \rmK_R) \cong R$, we obtain that
\[
\ell_R(\Hom_R(M, \rmK_R)/[\fkm \Hom_R(M, \rmK_R)+\psi_{M, \rmK_R}])=\ell_R(R/\fkm) = 1.
\]

 (only if part): Suppose that the condition {\rm (a)} holds. Then, $A$ is Gorenstein. Since $\fkm \rmK_R + \tr_{\rmK_R}(M) = \rmK_R$, $M$ is a residually faithful $R$-module by Fact \ref{f311} and Nakayama's lemma. Let $a_1, a_2, \dots, a_d\in \fkm$ be a system of parameter of $R$, and set $\fkq=(a_1, a_2, \dots, a_d)$.  Then $M/\fkq M$ is a faithful $R/\fkq$-module by Fact \ref{f311}. Furthermore, since $(a_1, 0), (a_2, 0), \dots, (a_d, 0)$ is a system of parameter of $A$, $A/\fkq A\cong R/\fkq \times_{\ol{\varphi}} M/\fkq M$ is a Gorenstein ring where $\ol{\varphi}$ is the natural map induced from $\varphi$ (Proposition \ref{graded}(4) and Theorem \ref{CM}). Since $M/\fkq M$ is faithful, it follows that $M/\fkq M \cong \rmK_{R/\fkq}$ by the proof of Theorem \ref{Gor}(1)$\Rightarrow$(2). Thus, $M$ is a faithful maximal Cohen-Macaulay $R$-module of Cohen-Macaulay type $1$, that is, $M\cong \rmK_R$ by \cite[Proposition 3.3.13]{BH}.

(2): By Nakayama's lemma, $\ell_R(\Hom_R(M, \rmK_R)/[\fkm \Hom_R(M, \rmK_R)+\psi_{M, \rmK_R}])=0$ if and only if $\psi_{M, \rmK_R} = \Hom_R(M, \rmK_R)$. On the other hand, for all $x\in M$ and $k\in \rmK_R$, we have 
\[
\Im \psi_{x, k} = \{\varphi(x, y)k \mid y\in M\} \subseteq \fkm \rmK_R
\] 
since $\varphi(M\times M) \subseteq \fkm$. It follows that $\tr_{\rmK_R}(M)\subseteq \fkm \rmK_R$ when $\psi_{M, \rmK_R} = \Hom_R(M, \rmK_R)$. Therefore, the condition {\rm (b)} holds if and only if
\[
\psi_{M, \rmK_R} = \Hom_R(M, \rmK_R)\quad \text{ and } \quad \ell_R(\rmK_R/\fkm \rmK_R)=1.
\]
The latter condition is equivalent to saying that $R$ is Gorenstein. Hence, we can also replace $\rmK_R$ by $R$ in the former equation of the above conditions.
\end{proof}



\begin{rem}\label{rem4.15}
\begin{enumerate}[$(1)$]
\item
Let $a\in R$, $x\in M$ and consider the map $\psi_{x, a}\in \psi_{M, R}$. Since $\psi_{x, a}(y)=\varphi(x, y)a=\varphi(ax, y)$ for each $y\in M$, we have 
$$\psi_{M, R}=\{ \varphi(ax,-) \mid x\in M, \ a\in R \}=\{ \varphi(x, -) \mid x\in M \},$$ 
where $\varphi(x, -): M\to R$ ; $y\mapsto \varphi(x, y)$. Therefore, the equation $\Hom_R(M, R)=\psi_{M, R}$ means that every homomorphism in $\Hom_R(M, R)$ is in the form of $\varphi(x, -)$ for some $x\in M$.

\item The conditions of Theorems \ref{Gor}(2)(ii) and \ref{higherGor}(2)(ii) are equivalent in dimension zero. Suppose that $R$ is a Gorenstein ring and $M$ is a maximal Cohen-Macaulay $R$-module. Then, the equation $\Hom_R(M, R)=\psi_{M, R}$ implies that $M_{\varphi}=0$, and the converse also holds if $\dim R=0$.

Indeed, let $(-)^*$ denote the $R$-dual $\Hom_R(-, R)$.  
We consider the inclusion map $i: \psi_{M, R} \to M^*$ and the induced map $i^*: M^{**}\to (\psi_{M, R})^*$.
We also consider the canonical map $h: M\to M^{**}$, where $[h(x)](f)=f(x)$ for $x\in M$ and $f\in M^*$ (\cite[Theorem 3.3.10]{BH}).
Then, we obtain that 
$$\Ker \ i^*=\{ h(x) \mid x\in M_{\varphi} \}\cong M_{\varphi}.$$
It follows that $M_{\varphi}=0$ if and only if the map $i^*$ is injective. Therefore, if the map $i$ is bijective, then $M_{\varphi}=0$. 
The converse holds true if $\dim R=0$ since $\psi_{M, R}$ is also a maximal Cohen-Macaulay $R$-module.
\end{enumerate}

\end{rem}

As explained in Subsection \ref{subsection4.1}, our interest is in  the condition of \ref{higherGor}(2)(ii). 
Let $R$ be a Gorenstein ring of dimension $d>0$ and $M$ be a maximal Cohen-Macaulay $R$-module of rank $\ell>0$. If $\ell \ge 2$, by Theorem \ref{rank}(2), $\varphi=0$, that is, $A$ is the idealization in the sense of \cite{AW}. In this case, $A$ is Gorenstein if and only if $M\cong R$ by Theorem \ref{higherGor}. Thus, we assume that $\ell=1$. Then, $M= I$ for some ideal $I$ of $R$ since $M$ is a torsionfree $R$-module of rank one. Then $\varphi$ has the following form: $\varphi: I\times I \to R$ ; $(x, y)\mapsto \alpha xy$ for some $\alpha \in \rmQ(R)$  by Theorem \ref{rank}(1). 

With these assumptions, we characterize the Gorenstein property of $A=R\times_{\varphi} I$ as follows.

\begin{cor}\label{cc418}
Suppose that $R$ is a Gorenstein ring of dimension $d>0$. Let $I$ be an ideal of $R$ containing a non-zerodivisor of $R$ such that $I$ is a maximal Cohen-Macaulay $R$-module but $I\ncong R$. Set $\varphi: I\times I \to R; (x, y)\mapsto \alpha xy$ for some $\alpha \in \rmQ(R)$. Then the following conditions are equivalent. 
\begin{enumerate}[$(1)$]
\item $A=R\times_{\varphi} I$ is Gorenstein. 
\item $\alpha I=R:I$.
\item $\alpha \in \rmQ(R)^\times$, $\alpha^{-1} \in I$, and $I=\left(\alpha^{-1}\right):_R I$.
\end{enumerate}

\end{cor}
\begin{proof}
(1) $\Leftrightarrow$ (2): Since there exists the following commutative diagram
$$
\xymatrix{
 \psi_{I, R} \ar[d]^{\wr} & \subseteq  & \Hom_R(I, R)  \ar[d]^{\wr}  \\
 \alpha I & \subseteq & R:I 
}
$$
of $R$-modules, we have that $A$ is Gorenstein if and only if $\alpha I=R:I$ by Theorem \ref{higherGor} and Remark \ref{rem4.15}. 

(2) $\Rightarrow$ (3): Since $1\in R:~I=\alpha I$, we have $\alpha \in \rmQ(R)^\times $ and $\alpha^{-1} \in I=\alpha^{-1}\cdot (R:I)=\left(\alpha^{-1}\right):_R I$. 

(3) $\Rightarrow$ (2): This is clear since $\alpha^{-1}\cdot (R:I)=\left(\alpha^{-1}\right):_R I$.
\end{proof}

In the case of dimension one, we can construct Gorenstein rings $A=R\times_\varphi I$ by using the notion of good ideals in the sense of \cite{GIW}. In one-dimensional Gorenstein local ring $(R, \fkm)$, an $\fkm$-primary ideal $I$ is called a {\it good ideal} if $I^2=aI$ and $I=(a):_R I$ for some parameter ideal $(a)\subseteq I$ (this definition is not the usual one, but equivalent to it under our assumption; see \cite[Proposition (2.2)]{GIW}). It is also known that there exists a one-to-one correspondence between the set of good ideals and the set of Gorenstien birational extensions of $R$ (\cite[Theorem (4.2)]{GIW}).

\begin{cor}
Suppose that $R$ is a Gorenstein local ring of dimension one. Let $C$ be a Gorenstien ring such that $R\subsetneq C\subseteq \rmQ(R)$ and $C$ is finitely generated as an $R$-module. 
Then, $I:=R:C$ is a good ideal of $R$; hence, $R\times_\varphi I$ is Gorenstein, where $a\in I$ such that $I^2=aI$, $I=(a):_R I$ and  
\[
\varphi: I \times I \to R; (x, y)\mapsto a^{-1} xy.
\]
\end{cor}

\begin{proof}
We note that $I$ is a good ideal of $R$ by \cite[proof of Proposition (2.2)]{GIW}.
We then apply Corollary \ref{cc418} with $\alpha=a^{-1}$.
\end{proof}

We note here the Gorenstein property of $A=R\times_{\varphi} \fkm$ in dimension one.

\begin{cor}\label{c417}
Suppose that $R$ is a Gorenstein ring of dimension one, but not a discrete valuation ring. Then the following conditions are equivalent. 
\begin{enumerate}[$(1)$]
\item $A=R\times_{\varphi} \m$ is Gorenstein.
\item $\alpha \in \rmQ(R)^\times$ and $\alpha^{-1} \in \m$.
\item $\rme(R)=2$, $\alpha \in \rmQ(R)^\times$, $\alpha^{-1}\in \m$ , and $\m^2 = \alpha^{-1} \m$.
\end{enumerate}
\end{cor}

\begin{proof}
(1) $\Rightarrow$ (2): This follows from Corollary \ref{cc418}(1) $\Rightarrow$ (3). 

(2) $\Rightarrow$ (1), (3): Suppose that $\alpha \in \rmQ(R)^\times$ and $\alpha^{-1} \in \m$. Since $\alpha \m^2= \varphi(\m, \m)\subseteq \m$, we have $\m^2\subseteq \alpha^{-1}\m\subseteq \m^2$. It follows that $\m=\left( \alpha^{-1} \right):_R \m$ and $\m^2=\alpha^{-1}\m$. Hence, $\rme(R)=2$ since $R$ is Gorenstein. The assertion (1) also follows from Corollary \ref{cc418}(3) $\Rightarrow$ (1). 

(3) $\Rightarrow$ (2): This is clear.
\end{proof}

\begin{rem}\label{r418}
If $R$ is a discrete valuation ring, then $R\times_{\varphi} \m$ is always Gorenstein because $\m\cong R=\rmK_R$.

\end{rem}


\section{Appendix 2: The regularity of $\mathbb{Z}_2$-graded local rings}\label{s7}

In this appendix, we note the regularity of $A$. We should emphasize that the idealization $R\times_0 M$ of a nonzero $R$-module $M$ is never to be reduced, and thus not a regular ring. 
We maintain Setup \ref{setup2}.


\begin{lem}\label{embdim}
The equality 
\[
v(A)=\mu_R(\fkm/\varphi(M, M)) + \mu_R(M)
\] 
holds. Therefore, if $\fkm=(c_1, c_2, \dots, c_s) +~\varphi(M, M)$ and $M=(x_1, x_2, \dots, x_t)$, where $s=\mu_R(\fkm/\varphi(M, M))$ and $t=\mu_R(M)$, then 
\[
(c_1, 0), \dots, (c_s, 0),  (0,x_1), \dots, (0,x_t)
\]
is a system of minimal generators of $\fkn=\fkm\times M$.
\end{lem}

\begin{proof}
By Corollary \ref{localc}, $\fkn=\fkm\times M$ is the maximal ideal of $A$. Hence, 
\begin{align*}
v(A)=&\ell_A((\fkm\times M)/(\fkm\times M)^2)=\ell_A((\fkm\times M)/[(\fkm^2 + \varphi(M, M))\times \fkm M])\\
=& \ell_R(\fkm/(\fkm^2 + \varphi(M, M))) + \ell_R(M/\fkm M),
\end{align*}
where the last equality follows from the isomorphism $R/\fkm \cong A/(\fkm \times M)$.
\end{proof}

\begin{cor}\label{r62}
$A$ is a regular local ring if and only if  $d=\mu_R(\fkm/\varphi(M, M)) + \mu_R(M)$. In particular, we have the following.
\begin{enumerate}[\rm(1)] 
\item If $A$ is regular, then $1\le \mu_R(M) \le d$.
\item If $d=0$, $A$ is never to be regular. 
\end{enumerate}
\end{cor}


In what follows, we suppose that $d>0$ unless otherwise noted. 
We note that there exist examples of regular rings $A$ such that $\mu_R(M)$ can be arbitrary among the interval $1\le \mu_R(M) \le d$.

\begin{ex}\label{exr1}
Let $R$ be a regular local ring of dimension $d$. Set $\varphi: R \times R \to R$ such that $\varphi(1,1)\not\in \fkm^2$. Then $A=R\times_{\varphi} R$ is regular.
\end{ex}

\begin{proof}
This follows from the equation $v(A)=\mu_R(\fkm/\varphi(R, R)) + \mu_R(R) = (d-1) + 1=d$.
\end{proof}

\begin{ex}\label{exsec5}
Let $K$ be a field. Let $m, n$ be non-negative integers such that $n\ge 2$. Let $K[[x_1, x_2, \dots, x_{n-1}, y_1, y_2, \dots, y_m, t]]$ be the formal power series ring over $K$. Set 
\[
R=K[[t, x_i t, x_j  x_k t, y_\ell \mid 1 \le i,j,k \le n-1, 1 \le \ell \le m]].
\]
Let 
\[
I=(t, x_i t \mid 1 \le i \le n-1)
\]
and $\varphi: I \times I \to R; (f, g) \mapsto \frac{fg}{t}$. Then $A=R\times_{\varphi} I$ is a regular local ring of dimension $m+n$ such that $\mu_R(\fkm/\varphi(I, I)) =m$ and  $\mu_R(I)=n$.
\end{ex}

\begin{proof}
It is straightforward to check that $\varphi(I, I)=\frac{I^2}{t}=(t, x_i t, x_j  x_k t \mid 1 \le i,j,k \le n-1)$. Hence, 
\[
v(A)=\mu_R(\fkm/\varphi(I, I)) + \mu_R(I)  =m+n,
\]
where $\fkm$ denotes the maximal ideal of $R$, by Lemma \ref{embdim}. On the other hand, $\dim A=\dim R=m+n$ by \cite[Theorem 6.1.7]{BH} (or localize with the multiplicative set $\{t^p \mid p\ge 0\}$). It follows that $A$ is regular.
\end{proof}

\begin{prop}\label{p55}
Suppose that $A$ is a regular ring. 
Then  we have the following.
\begin{flalign*}
&(1) \ \text{$R$ is a Cohen-Macaulay domain. } \quad (2) \ \text{$R/\varphi(M, M)$ is regular. }&&\\
&(3) \ \text{$M$ is a maximal Cohen-Macaulay $R$-module and $M\cong I$ for some ideal $I$ of $R$. } 
\end{flalign*}
\end{prop}

\begin{proof}
(1): $R$ is Cohen-Macaulay by Theorem \ref{CM}. Since $A$ is a domain and there is an injective map $R \to A;\ a\mapsto (a, 0)$, $R$ is a domain.

(2): Let $x_1, x_2, \ldots, x_s\in M$ be a minimal generators of $M$. Then, since $(0, x_1), (0, x_2), \ldots, (0, x_s)$ is a part of minimal generators of the maximal ideal $\m\times M$ (Lemma \ref{embdim}) and 
$$\left( (0, x_1), (0, x_2), \ldots, (0, x_s) \right) A=\left( 0\times M \right)A=\varphi(M, M)\times M, $$
we obtain that $A/\left(\varphi(M, M)\times M\right) \cong R/\varphi(M, M)$ is also a regular local ring.

(3): $M$ is a maximal Cohen-Macaulay $R$-module by Theorem \ref{CM}. Since $R$ is a Cohen-Macaulay domain by (1), $M$ has a positive rank $\ell$. If $\ell\ge 2$, then $A$ is the idealization of $M$ by Theorem \ref{rank}(1). This is a contradiction since the nilradical of the idealization $A$ contains $(0)\times M$. Hence, $\ell=1$. Since $M$ is torsionfree, this concludes that $M$ can be embedded into $R$ (see, for example, \cite[Excercise 1.4.18]{BH}).
\end{proof}

When $1\le \mu_R(M)\le 2$, we can characterize the regularity of $A$ via the triad $(R, M, \varphi)$.

\begin{cor}
Suppose that $\mu_R(M)=1$. The following are equivalent.
\begin{enumerate}[\rm(1)] 
\item $A$ is a regular local ring.
\item $R$ is a regular local ring, $M\cong R$, and $\varphi(\mathbf{e}, \mathbf{e})\not\in \fkm^2$, where $\mathbf{e}$ denotes a free basis of $M$.
\end{enumerate}
\end{cor}

\begin{proof}
The implication $(2)\Rightarrow(1)$ follows from Example \ref{exr1}.

$(1)\Rightarrow(2)$: Since $\mu_R(M)=1$ and $R$ is domain, we obtain that $M\cong R$ by Proposition \ref{p55}$(3)$. Let $a=\varphi(\mathbf{e}, \mathbf{e})\in\m$, where $\mathbf{e}$ denotes a free basis of $M$. Since 
$$\mu_R(\m/(\m^2+aR))=\mu_R(\m/(\m^2+\varphi(M, M)))=v(A)-\mu_R(M)=d-1$$
by Lemma \ref{embdim},
we have $v(R)\le d$. This  induces that $R$ is regular and $a=\varphi(\mathbf{e}, \mathbf{e})\notin \m^2$.
\end{proof}

Recall that $R$ is {\it hypersurface} if $v(R)\le d+1$.

\begin{thm}
Suppose that $\mu_R(M)= 2$. The following are equivalent.
\begin{enumerate}[\rm(1)] 
\item $A$ is a regular local ring.
\item The following hold true.
\begin{enumerate}[\rm(i)] 
\item $R$ is a non-regular hypersurface domain with $d \ge 2$. 
\item $M\cong I$ for some ideal $I$ of $R$ such that $\mu_R(I)=2$. 
\item $\varphi(M, M)=(a_1, a_2, a_3)$, where $a_1, a_2, a_3$ is a part of minimal generators of $\fkm$.
\end{enumerate}
\end{enumerate}

\end{thm}

\begin{proof}
$(2)\Rightarrow(1)$: Conditions in (2) implies that
$$v(A)=\mu_R(\fkm/\varphi(M, M)) + \mu_R(M)=[(d+1)-3]+2=d.$$

$(1)\Rightarrow(2)$: By Lemma \ref{embdim}, we have $0\le \mu_R(\fkm/\varphi(I, I))= v(A)-\mu_R(I)= d-2.$ Thus, $d\ge 2$. (ii) follows from Proposition \ref{p55}(3). Therefore, we may assume that $M= I$ for some ideal $I$ of $R$ such that $\mu_R(I)=2$, and there exists $\alpha \in \rmQ(R)$ such that $\varphi(x, y)=\alpha xy$ for any $x, y\in I$ by Theorem \ref{rank}(2). Note that $\mu_R(\varphi(I, I))=2$ or $3$ since $\varphi(I, I)=\alpha I^2\cong I^2$ and $\mu_R(I)=2$.

Next, we prove (i) and (iii) when $d=2$. We then have $\varphi(I, I)=\m$ since $\mu_R(\fkm/\varphi(I, I))=0$. We need the following claim.

 
 
\begin{claim}\label{claim3}
Let $J$ be an ideal of a local ring $S$. If $\mu_S(J)=\mu_S(J^2)=2$, then $J^2=aJ$ for some $a\in J$.
\end{claim} 
\begin{proof}[Proof of Claim \ref{claim3}]
Let $J=(a, b)$. Since $\mu_S(J^2)=2$, $J^2$ coincides with one of the following:
\begin{center}
$(a^2, ab)=aJ$, $(b^2, ab)=bJ$, or $(a^2, b^2)$.
\end{center} 
In the first two cases, there is nothing to say.
 Suppose that $J^2=(a^2, b^2)$, $J^2\neq aJ$, and $J^2\neq bJ$. Since $ab\in J^2=(a^2, b^2)$, we write $ab=ca^2+db^2$ with $c, d\in R$. If $c\in R^\times$ (resp. $d\in R^\times$), then $a^2\in (b^2, ab)$ (resp. $b^2\in (a^2, ab)$), so that $J^2=bJ$ (resp. $J^2=aJ$). Therefore, we have $c, d\in \m$. Let $a_1=a-b$. Then, 
$$ab=(a_1+b)b=a_1b+b^2\quad \text{and} \quad ca^2+db^2=c(a_1+b)^2+db^2=ca_1^2+2ca_1b+(c+d)b^2.$$
Since $ab=ca^2+db^2$, we have $(1-(c+d))b^2\in (a_1^2, a_1b)$. This implies that $b^2\in (a_1^2, a_1b)$, because $c, d\in \m$. Consequently, since $J=(a_1, b)$, we obtain that $J^2=(a_1^2, a_1b)=a_1J$, as desired.
\end{proof}

Suppose that $d=2$ and $\mu_R(\m)=\mu_R(\varphi(I, I))=2$. Then, $R$ is a two-dimensional regular local ring. 
 We also have 
$$\m=\varphi(I, I)=\alpha I^2\cong I^2=aI\cong I$$ 
for some $a\in I$ by Claim \ref{claim3}. It follows that $\m=\beta I$ for some $\beta \in \rmQ(R)$. Hence, we have $\m^2=\beta^2I^2\cong I^2\cong \m$, which implies that $\mu_R(\m^2)=\mu_R(\m)=2$. This contradicts that $\m$ is a parameter ideal of $R$.
 Therefore,  $\mu_R(\varphi(I, I))=\mu_R(\m)=3$, which induces that $R$ is a non-regular hypersurface domain and $\varphi(I, I)=\m=(a_1, a_2, a_3)$.    

Suppose that $d\ge 3$. Since $\mu_R(\fkm/\varphi(I, I))= d-2$,
we choose $b_1, b_2, \dots, b_{d-2}\in \m$ such that $\m=(b_1, b_2, \dots, b_{d-2}) + \varphi(I, I)$. 
Then, $b_1, b_2, \dots, b_{d-2}$ is a part of minimal generators of $\m$, and  $(b_1, 0), \dots, (b_{d-2}, 0)$ is a part of minimal generators of $\fkn=\m\times M$ by Lemma \ref{embdim}. 
Consider $\overline{A}:= A/((b_1, 0), \dots, (b_{d-2}, 0)) \cong R/\fkb \times_{\ol{\varphi}} M/\fkb M$, where $\fkb=(b_1, b_2, \dots, b_{d-2})$. Then, $\ol{A}$ is a regular local ring of dimension two. By the case of $d=2$, we have $R/\fkb$ is a non-regular hypersurface domain and $\mu_R((\varphi(I, I)+\fkb)/\fkb)=3$. This implies that $R$ is a non-regular hypersurface domain and $\mu_R(\varphi(I, I))=3$, since $\mu_R(\varphi(I, I))=2$ or $3$. 
Since $\m=\varphi(I, I)+\fkb$ and $v(R)=d+1$, we have $\varphi(I, I)=(a_1, a_2, a_3)$, where $a_1, a_2, a_3$ is a part of minimal generators of $\fkm$.
\end{proof}





By the inequality $1\le \mu_R(M) \le d$ in Corollary \ref{r62}(1), we conclude the following. 

\begin{cor}\label{p59}
The following hold true.
\begin{enumerate}[\rm(1)] 
\item If $d=0$, then $A$ is never to be regular.
\item If $d=1$, then $A$ is regular if and only if $R$ is a discrete valuation ring, $M\cong R$, and $\varphi(M, M)=\fkm$.
\item If $d=2$, then $A$ is regular if and only if one of the following hold.
\begin{enumerate}[\rm(i)] 
\item $R$ is a regular local ring, $M\cong R$, and $\varphi(\mathbf{e}, \mathbf{e})\not\in \fkm^2$, where $\mathbf{e}$ denotes a free basis of $M$.
\item The following hold true.
\begin{enumerate}[\rm(a)] 
\item $R$ is a non-regular hypersurface domain. 
\item $M\cong I$ for some ideal $I$ of $R$ such that $\mu_R(I)=2$ and $\mu_R(I^2)=3$.
\item $\varphi(M, M)=\fkm$.
\end{enumerate}
\end{enumerate}
\end{enumerate}
\end{cor}







\begin{acknowledgments}
The authors are grateful to Yuta Takahashi.
\end{acknowledgments}




\begin{thebibliography}{99}

\bibitem{AW}
{\sc D. D. Anderson; M. Winders}. Idealization of a module. {\em J. Commut. Algebra} {\bf 1} (2009), no. 1, 3--56.




\bibitem{BF}
{\sc V. Barucci; R. Fr\"{o}berg}. One-dimensional almost Gorenstein rings. {\em J. Algebra} {\bf 188} (1997), no. 2, 418--442.




\bibitem{BV}
{\sc J. P. Brennan; W. V. Vasconcelos}. On the structure of closed ideals. {\em Math. Scand.} {\bf 88} (2001), no. 1, 3--16.

\bibitem{BH}
{\sc W. Bruns; J. Herzog}. Cohen-Macaulay rings. Cambridge Studies in Advanced Mathematics, 39. {\em Cambridge University Press, Cambridge}, 1993. {\rm xii}+403 pp.

\bibitem{C}
{\sc M. Claudia}. Cohen-Macaulay and Gorenstein finitely graded rings. {\em Rendiconti del Seminario Matematico della Universit\`{a} di Padova} {\bf 79} (1988), 123--152.







\bibitem{F}
{\sc R. Fossum}. Commutative extensions by canonical modules are Gorenstein rings. {\em Proc. Amer. Math. Soc.} {\bf 40} (1973), 395--400.

\bibitem{GIW}
{\sc S. Goto; S.-i. Iai; K.-i. Watanabe}. Good ideals in Gorenstein local rings. {\em Trans. Amer. Math. Soc.} {\bf 353} (2001), no. 6, 2309--2346.

\bibitem{GIK2}
{\sc S. Goto; R. Isobe; S. Kumashiro}. Correspondence between trace ideals and birational extensions with application to the analysis of the Gorenstein property of rings. {\em J. Pure Appl. Algebra} {\bf 224} (2020), no. 2, 747--767.

\bibitem{GKL}
{\sc S. Goto; S. Kumashiro; N. T. H. Loan}. Residually faithful modules and the Cohen-Macaulay type of idealizations. {\em J. Math. Soc. Japan} {\bf 71} (2019), no. 4, 1269--1291.

\bibitem{GK}
{\sc S. Goto; S. Kumashiro}. On generalized Gorenstein local rings, arXiv:2212.12762.

\bibitem{GMP}
{\sc S. Goto; N. Matsuoka; T. T. Phuong},  Almost Gorenstein rings {\em J. Algebra}. {\bf 379} (2013), 355--381.





\bibitem{GTT}
{\sc S. Goto; R. Takahashi; N. Taniguchi}. Almost Gorenstein rings—towards a theory of higher dimension. {\em J. Pure Appl. Algebra} {\bf 219} (2015), no. 7, 2666--2712.



\bibitem{HHS}
{\sc J. Herzog; T. Hibi; D. I. Stamate}. The trace of the canonical module. {\em Israel J. Math.} {\bf 233} (2019), no. 1, 133--165.



\bibitem{HK2}
{\sc J. Herzog; S. Kumashiro}. Upper bound on the colength of the trace of the canonical module in dimension one. {\em Arch. Math. (Basel)} {\bf 119} (2022), no. 3, 237--246.


\bibitem{HK}
{\sc J. Herzog; E. Kunz}, Der kanonische Modul eines Cohen-Macaulay-Rings, Lecture Notes in Mathematics, {\bf 238} (1971), Springer-Verlag.





\bibitem{HS}
\textsc{C. Huneke; I. Swanson}, Integral closure of ideals, rings, and modules. London Mathematical Society Lecture Note Series, 336. \textit{Cambridge University Press, Cambridge}, 2006. xiv+431 pp.






\bibitem{Kob}
{\sc T, Kobayashi}, Syzygies of Cohen-Macaulay modules over one dimensional Cohen-Macaulay local rings. {\em Algebr. Represent. Theory} {\bf 25} (2022), no. 5, 1061--1070.


\bibitem{Kum}
{\sc S. Kumashiro}. The reduction number of canonical ideals. {\em Comm. Algebra} {\bf 50} (2022), no. 11, 4619--4635.

\bibitem{Kum2}
{\sc S. Kumashiro}. When are trace ideals finite?. {\em Mediterr. J. Math.} {\bf 20} (2023), no. 5, Paper No. 278, 17 pp.


\bibitem{Lin}
{\sc H. Lindo}. Trace ideals and centers of endomorphism rings of modules over commutative rings. {\em J. Algebra} {\bf 482} (2017), 102--130.

\bibitem{LP}
{\sc H. Lindo; N. Pande}. Trace ideals and the Gorenstein property. {\em Comm. Algebra} {\bf 50} (2022), no. 10, 4116--4121.




\bibitem{Mat}
{\sc H. Matsumura}. Commutative ring theory. Translated from the Japanese by M. Reid. Cambridge Studies in Advanced Mathematics, 8. {\em Cambridge University Press, Cambridge}, 1986. {\rm xiv}+320 pp. 

\bibitem{R}
{\sc I. Reiten}. The converse to a theorem of Sharp on Gorenstein modules. {\em Proc. Amer. Math. Soc.} {\bf 32} (1972), 417--420.



\end{thebibliography}
\end{document}